\documentclass[11pt]{amsart}
\usepackage{amsmath}
\usepackage{enumerate}
\usepackage{amssymb}
\usepackage{pb-diagram}
\usepackage{enumerate}
\usepackage{graphicx}
\usepackage[active]{srcltx}
\usepackage[lite]{amsrefs}

\newtheorem{theorem}{Theorem}[section]
\newtheorem{thm}[theorem]{Theorem}
\newtheorem{prop}[theorem]{Proposition}

\newtheorem{claim}[theorem]{Claim}
\newtheorem{fact}[theorem]{Fact}

\newtheorem{lemma}[theorem]{Lemma}

\theoremstyle{definition}
\newtheorem{defn}[theorem]{Definition}

\newtheorem{example}[theorem]{Example}

\theoremstyle{remark}
\newtheorem{rmk}[theorem]{Remark}
\newtheorem{remark}[theorem]{Remark}

\newcommand{\Z}{\mathbb{Z}}
\newcommand{\N}{\mathbb{N}}

\newcommand{\R}{\mathbb{R}}

\newcommand{\CM}{\mathcal M}
\newcommand{\sub}{\subseteq}
\newcommand{\CN}{\mathcal N}
\newcommand{\ra}{\rangle}
\newcommand{\la}{\langle}

\newcommand{\noi}{\noindent}
\newcommand{\CH}{\mathcal H}

\newcommand{\CB}{\mathcal B}
\newcommand{\CK}{\mathcal K}
\newcommand{\Vdef}{$\bigvee$-definable }

\newcommand{\CV}{\mathcal V}

\newcommand{\CU}{\mathcal U}

\newcommand{\lgdim}{\ensuremath{\textup{lgdim}}}

\newcommand{\Lrarr}{\ensuremath{\Leftrightarrow}}

\newcommand{\wh}{\widehat}
\newcommand{\wH}{\widehat \CH}

\newcommand{\bb}[1]{\ensuremath{\mathbb{#1}}}

\newcommand{\al}{\alpha}

\newcommand{\Lam}{\ensuremath{\Lambda}}

\newcommand{\cal}[1]{\ensuremath{\mathcal{#1}}}

\title[Definable groups]{Definable groups
as homomorphic images of semi-linear and field-definable groups}
\author{Pantelis~E.~Eleftheriou}
\address{CMAF, Universidade de Lisboa,
Av. Prof. Gama Pinto 2, 1649-003 Lisboa, Portugal} \email{pelefthe@uwaterloo.ca}
\thanks{The first author was supported by the Funda\c{c}\~ao para a Ci\^encia e a Tecnologia
grants SFRH/BPD/35000/2007 and PTDC/MAT/101740/2008}

\author{Ya'acov Peterzil}
\address{Department of Mathematics, University of Haifa, Haifa, Israel}
\email{kobi@math.haifa.ac.il}

\begin{document}

\date{\today}

\begin{abstract}
We analyze definably compact groups  in o-minimal expansions of ordered groups as a
combination of semi-linear groups and groups definable in o-minimal expansions of
real closed fields. The analysis involves structure theorems about their locally
definable covers.  As a corollary, we prove the Compact Domination Conjecture in
o-minimal expansions of ordered groups.
\end{abstract}
 \maketitle

\section{Introduction}
This is the second of two papers (originally written as one)
analyzing groups definable in o-minimal expansions of ordered groups. The ultimate goal of this project
is to reduce the analysis of such groups  to semi-linear groups and to groups definable in o-minimal
 expansions of real closed fields. Such a reduction was proposed in Conjecture 2 from \cite{pet-sbd}
 and a first step towards it was carried out in \cite{el-sbd}. 

In the first paper (\cite{ep-defquot}) we established conditions under which locally definable groups
 have definable quotients of the same dimension.
 In this paper, we carry out the aforementioned reduction for definably compact groups by first stating a
  structure theorem for the universal cover $\widehat G$ of a definable group $G$ (Theorem \ref{thm1}). We  describe
  $\widehat G$
  as an extension of a locally definable group $\CU$ in an o-minimal expansion of a real closed field by a locally definable
  semi-linear group $\widehat H$.
   We then apply \cite[Theorem 3.10]{ep-defquot} and derive a stronger structure theorem
   (Theorem \ref{thm2}), replacing the above $\CU$ by {\em a definable group}. We expect that the second theorem will be useful when
    reducing questions for definable groups to groups in the semi-linear and field settings.
    We illustrate this effect by applying our second theorem to conclude the Compact
    Domination Conjecture in o-minimal expansions of ordered groups (Theorem \ref{thm-cdom} below).

Let us provide the details.

\subsection{The setting} We let $\cal{M}=\langle M, <, +, 0,
\dots\rangle$ be an o-minimal expansion of an ordered group. When
$\CM$ expands a real closed field (with $+$ not necessarily one
of the field operations) there is strong compatibility of
definable sets with the field structure. For example, each
definable function is piecewise differentiable with respect to the
field structure. Other powerful tools, such as the triangulation
theorem, are available as well (\cite{vdd}). At the other end,
when $\CM$ is a linear structure, such as a reduct of an ordered
vector space over an ordered division ring, then every definable
set is \emph{semi-linear}.

By the Trichotomy Theorem for o-minimal structures there is a third possibility (see
\cite{pest-tri}), where there is a definable real closed field $R$ on some interval
in $M$, and yet the underlying domain of $R$ is necessarily a bounded interval and
not the whole of $M$. Such a structure is called \emph{semi-bounded} (and
non-linear), and definable sets in this case turn out to be a combination of
semi-linear sets and sets definable in o-minimal expansions of fields (see
\cite{ed-sbd}, \cite{pet-sbd}, \cite{el-sbd}). An important example is the expansion
of the ordered vector space $\la \R; <,+, x\mapsto ax\ra_{a\in \R}$ by
all bounded semialgebraic sets. Most of our  work is intended for a semi-bounded structure which is non-linear.\\

\emph{We assume in the rest of this paper, and unless stated otherwise, that $\CM=\la M,<,+,\cdots\ra$ is a sufficiently saturated
o-minimal expansion of an ordered group.}\vskip.2cm

\subsection{Short sets and long dimension.} Following \cite{pet-sbd}, we call an
element $a\in M$ {\em short} if either $a=0$ or the interval $(0,a)$ supports a
definable real closed field; otherwise $a$ is called {\em tall}. An element of $M^n$
is called {\em short} if all its coordinates are short. An interval $[a,b]$ is
called {\em short} if $b-a$ is short, and otherwise it is called {\em long}. A
definable set $X\sub M^n$ is called {\em short} if it is in definable bijection with
a subset of $I^n$ for some short interval $I$. The image of a short set under a
definable map is short. As is shown in \cite{ed-sbd}, $\CM$ is semi-bounded if and
only if all unbounded rays $(a,+\infty)$ are long. However, a semi-bounded and
sufficiently saturated $\cal M$ also has bounded intervals which are long.

Following \cite{el-sbd} (see also Section \ref{sec-sbd} below), we say that {\em the
long dimension} of a definable $X\sub M^n$, $\lgdim(X)$, is the maximum $k$ such
that $X$ contains a  definable homeomorphic image of $I^k$, for some long interval
$I$ (the original definition of $\lgdim(X)$ was given in terms of cones, see Section
3 below, but it is not hard to see the equivalence of the two). The results in
\cite{el-sbd} show that every definable subset of $M^n$ can be decomposed into
``long cones'' and as a result it follows that a definable $X\sub M^n $ is short if
and only if $\lgdim(X)=0$.  We call $X$ {\em strongly long} if $\lgdim(X)=\dim(X)$;
this is for example the case with a cartesian product of long intervals.  Note that
all these notions are invariant under definable bijections.

Roughly speaking, strongly long sets and short sets are ``orthogonal'' to each
other. The idea is that the structure which $\CM$ induces  on short sets comes from
an o-minimal expansion of a real closed field, while the structure induced on
strongly long sets is closely related to the semi-linear structure. More precisely,
if $p(x)$ is a complete type over $A$ such that every formula in $p(x)$ defines a
strongly long set then its semi-linear formulas determine the type.  This is a
result which will not be used in this paper, but its proof is straightforward.
Indeed, the aforementioned decomposition from \cite{el-sbd} implies, in particular,
that every strongly long definable set $X$ of dimension $k$ is a union of a strongly
long $k$-dimensional semi-linear set and a definable set whose long dimension is
smaller than $k$. Both sets are definable over the same set of parameters as $X$. It
follows that $p(x)$ is determined by the semi-linear formulas.

We will see in examples (Section \ref{nonextension}) that
 the analysis of definable groups forces us to use the language of $\bigvee$-definable groups, so we recall some definitions.

\subsection{\Vdef and locally definable sets.} Let $\cal M$ be a $\kappa$-saturated,
 not necessarily o-minimal, structure. By \emph{bounded} cardinality we mean cardinality smaller than $\kappa$. We alert the reader that there is a second use of the word ``bounded" throughout this paper. Namely, a subset of $M^n$ is \emph{bounded} if it is contained in some cartesian
product of bounded intervals. It will always be clear from the context what we mean.

A \emph{\Vdef group}  is a group $\la \CU,\cdot \ra$ whose universe is a directed
union
  $\CU=\bigcup_{i\in I} X_i$ of
definable subsets of $M^n$ for some fixed $n$  (where $|I|$ is bounded) and for every
$i,j\in I$, the restriction of group multiplication to $X_i\times X_j$ is a
definable function (by saturation, its image is contained in some $X_k$). Following
\cite{ed1}, we say that $\la \CU,\cdot \ra$ is \emph{locally definable} if $|I|$ is
countable. In this paper, we consider exclusively locally definable groups. We are
mostly interested in {\em definably generated} groups, namely \Vdef groups which are
generated as a group by a definable subset. These groups are of course locally
definable. An important example of such groups is the universal cover of a definable
group (see \cite{edel2}). In \cite{HPP} a similar notion is introduced, of an
Ind-definable group.

A map $\phi:\CU\to \CH$ between \Vdef (locally definable) groups is called
{\em \Vdef (locally definable)}  if for every definable $X\sub \CU$ and $Y\sub \CH$,
$graph(\phi)\cap (X\times Y)$ is a definable set. Equivalently, the restriction of $\phi$  to any definable set is definable.

In an o-minimal expansion of an ordered group, a \Vdef group $\CU$ is called {\em
short} if $\CU$ is given as a bounded union of definable short sets. If
$\CU=\bigcup_{i\in I} X_i$ then we let $\lgdim(\CU)=max_i(\lgdim (X_i))$. We say
that $\CU$ is {\em strongly long} if $\dim(\CU)=\lgdim(\CU)$.\medskip

We are now ready to state the main results of this paper. Note that in
 the special case where $\cal M$ expands a real closed field, the
results below become trivial (since in this case all definable sets are short), and
in the case where $\cal M$ is semi-linear, they reduce to the main theorem from
\cite{ElSt} (since in this case every definable short set is finite).

\subsection{The universal cover of a definably compact group}
We first note (see \cite[Lemma 7.1]{pet-sbd}) that every definably compact group in
a semi-bounded structure is necessarily bounded; namely, it is contained in some cartesian
product of bounded intervals.

\begin{thm}\label{thm1} Let $G$ be a  definably compact, definably connected group of long dimension
$k$ and  let $\widehat F:\widehat G\to G$ be the universal cover of $G$. Then there
exist an open, connected subgroup $\widehat \CH\sub \la M^k,+\ra$, generated by a
semi-linear set of long dimension $k$, and a locally definable embedding $i:\widehat
\CH\to \widehat G$, with $i(\wH)$ central in $G$,  such that $\CU=\widehat
G/i(\widehat \CH)$  is generated by a short definable set. Namely, we have the
following exact sequence with locally definable maps $i$, $\pi$ and $\widehat F$:
$$
\begin{diagram}
\node{0}\arrow{e}\node{\widehat\CH} \arrow{e,t}{i}\node{\widehat G} \arrow{s,r}{\widehat F}\arrow{e,t}{\pi}\node{\CU} \arrow{e} \node{0}\\
\node[3]{G}
\end{diagram}
$$

If we let $\CH=\wh F(i(\wH))$, then $\CH$ is the largest connected, strongly long,
locally definable subgroup of $G$, namely it contains every other such group.

\end{thm}

\noi{\bf Question } In Section \ref{nonextension} we present various examples that illustrate this theorem.
 In all our known examples the universal cover $\wh G$ is the
direct sum of the groups $\wH$ and $\CU$ (rather then just an extension of $\CU$ by
$\wH$). Can $\widehat G$ always be realized as a direct sum of $\wH$ and
$\CU$?\\

\begin{remark}
1.  One immediate corollary of the above theorem is that every definably compact
group $G$ which is strongly long is definably isomorphic to a semi-linear group,
because in this case $\CH=G$.

2. Note that  when $G$ is abelian, we have $\ker(\wh F)\simeq\Z^{\dim G}$ (indeed,
by \cite[Corollary 1.5]{edel2}, we have  $ker(\wh F)\simeq \Z^l$, where the
$k$-torsion subgroups of $G$ satisfy $G[k]\simeq (\Z/k\Z)^l$. By \cite{pet-sbd}, we
have $l=\dim G$).

 3.  Note that since $\CU$ above is generated by a definable short set, there
is a definable real closed field $R$ such that $\CU$ is locally
definable in an o-minimal expansion of $R$. Indeed, let $X\sub
\CU$ be a definable set which generates $\CU$, and let $R$ be a
definable real closed field such that, up to an $\CM$-definable
definable bijection, $X$ is a subset of $R^m$. Let $\CN$ be the
structure which $\CM$ induces on $R$. Without loss of generality,
$0\in X$. We let $X_1=X$ and consider the equivalence relation on
$X\times X$ given by $(x,y)\sim (x',y')$ if $x-y=x'-y'$. Clearly,
$X\times X/\sim$ is in definable bijection with $X-X$. By
definable choice in $\CN$, there exists a definable set $Y$ in
$\CN$ and a definable bijection between $X\times X/\sim$ and $Y$.
Hence, in $\CM$ the sets $X-X$ and $Y$ are in definable bijection.
Now consider the definable embedding of $X$ into $X-X$ ($x\mapsto
x-0$), which induces an $\CN$-definable injection $f_1:X_1\to Y$.
We let
$$X_2=X_1\sqcup (Y\setminus f_1(X_1)).$$
The set $X_2$ is definable in $\CN$ and is in definable bijection
with $Y$ (so also with $X-X$). We also have $X_1\sub X_2$.

We similarly define $X_3$ in $\CN$ to be in definable bijection with $X-X+X$ and
such that $X_1\sub X_2\sub X_3$. We continue in the same way and obtain a locally
definable set $\bigcup_{n\in \mathbb N} X_n$ in $\CN$ that is in locally definable
bijection with $\CU$.
\end{remark}

\subsection{Covers by extensions of definable short groups}

In the next result we want to replace the locally definable group $\CU$ from Theorem
\ref{thm1} by a definable short group $\overline{K}$. Roughly speaking, it says that
$G$ is close to being an extension of a short definable group by a semi-linear
group, and the distance from being such a group is measured by the kernel of the map
$\overline{F'}$ below.

\begin{thm}\label{thm2}
Let $G$ be a definably compact, definably connected group of long dimension $k$.
Then  $G$ has
 a locally definable cover $\overline{F}:\overline{G}\to G$ with the following
 properties: there is an open subgroup $\wH\sub \la M^k,+\ra$,
 generated by a semi-linear set of long dimension $k$,  and a locally definable  embedding $i:\wH\to \overline{G}$,
 with $i(\wH)$ central in $\overline{G}$, such that
 $\overline{K}=\overline G/i(\wH)$ is a definably compact {\bf definable} short group.
 Namely, we have the following exact
sequence with locally definable  maps $i$, $\pi$ and $\overline F$:
$$
\begin{diagram}
\node{0}\arrow{e}\node{\wH} \arrow{e,t}{i}\node{\overline{G}}
\arrow{s,r}{\overline{F}}\arrow{e,t}{\pi}\node{\overline{K}}
\arrow{e} \node{0}\\
\node[3]{G}
\end{diagram}
$$
If we take $\CH\sub G$ as in Theorem \ref{thm1}, then there is also a locally
definable, central extension $G'$ of $\overline{K}$ by $\CH$, with a locally
definable homomorphism $F':G'\to G$.

When $G$ is abelian so is $\overline{G}$ and $\ker(\overline{F})\simeq \Z^k+F$,
for a finite group $F$.
\end{thm}
\vskip.2cm

It is at the passage from the locally definable group $\CU$ in Theorem \ref{thm1} to
the definable group $\overline{K}$ in Theorem \ref{thm2} that we use \cite[Theorem
3.10]{ep-defquot}.
\subsection{Compact Domination}

 The relationship between a definable group $G$ and the compact Lie group $G/G^{00}$ has
been the topic of quite a few  papers. In \cite{el-cdom}, \cite{HP}, \cite{HPP2} the
related so-called Compact Domination Conjecture was solved for semi-linear groups
and for groups definable in expansions of real closed fields. Using the above
analysis we can complete the proof of the conjecture for groups definable in
arbitrary o-minimal expansions of ordered groups (see Section \ref{CD-section} for
the original formulation of the conjecture).

\begin{thm}\label{thm-cdom} Let $G$ be a definably compact, definably
connected group. Let $\pi:G\to G/G^{00}$ denote the canonical homomorphism. Then,
$G$ is compactly dominated by $G/G^{00}$. That is, for every definable set $X\sub
G$, the set
$$\pi(X)\cap \pi(G\setminus X)$$
 has $\mathbf{Haar}$ measure $0$.
\end{thm}

\vspace{.5cm}

\subsection{Notation}\label{notation}
Let us finish this section with a couple of notational remarks. Given a group $\la
G, \cdot\ra$ and a set $X\sub G$, we denote, for every $n\in \bb{N}$,
\begin{equation}
X(n)=\overbrace{(XX^{-1}) \cdots  (XX^{-1})}^{n-\text{times}}\notag
\end{equation}

We assume familiarity with the notion of definable compactness. Whenever we write that a set is definably compact, or definably connected, we assume in particular that it is definable.

\vspace{.8cm}

 \noi {\bf Acknowledgements.} We thank Alessandro
Berarducci,  Mario Edmundo and Marcello Mamino for discussions which were helpful
during our work. We thank the referee for a very careful reading of the original
manuscript.

\section{Preliminaries I: locally definable groups, extensions of abelian groups, pushout and pullback}

As mentioned in the introduction, we work in a sufficiently saturated
o-minimal expansion of an ordered group a $\CM=\la M,<,+,\cdots\ra$. However, the only use of this assumption is to guarantee
 a strong version of elimination of imaginaries, which allows us to replace every definable quotient by a definable set. Any
structure in which this is true will be just as good here, or, if we are willing to
work in $\CM^{eq}$, then any o-minimal structure will work.

\subsection{Locally definable groups, compatible subgroups and  definable quotients}

\begin{defn}
(See \cite{ed1}) For a locally definable group $\CU$, we say that $\CV\sub \CU$ is
{\em a compatible subset of $\CU$} if for every definable $X\sub \CU$, the
intersection $X\cap \CV$ is a definable set (note that in this case $\CV$ itself is
a countable union of definable sets).
\end{defn}

Clearly, the only compatible locally definable subsets of
a definable group are the definable ones. Note that if $\phi:\CU\to \CV$ is a locally definable
homomorphism between locally definable groups then $\ker(\phi)$ is a compatible locally definable normal
subgroup of $\CU$. Compatible subgroups are used in order to obtain locally definable quotients. Together
with \cite[Theorem 4.2]{ed1}, we have:

\begin{fact}\label{edmundo} If $\CU$ is a locally definable group and $\CH\sub \CU$ a locally definable normal subgroup,
then $\CH$ is a compatible subgroup of $\CU$ if and only if there exists a locally
definable surjective homomorphism of locally definable  groups $\phi:\CU\to \CV$
whose kernel is $\CH$.
\end{fact}

If $\cal M$ is an o-minimal structure and $\CU\sub M^n$ is a locally definable group
then, by \cite[Theorem 4.8]{BaOt}, it can be endowed with a manifold-like topology
$\tau$, making it into a topological group. Namely, there is a countable collection
$ \{U_i:i\in I\}$ of definable subsets of $\CU$, whose union equals $\CU$, such that
each $U_i$ is in definable bijection with an open subset of $M^k$ ($k=\dim\CU$), and
the transition maps are continuous. Moreover the $U_i$'s and the transition maps are
definable over the same parameters as $\CU$. The group operation and group inverse
are continuous with respect to this induced topology.
 The topology $\tau$ is determined by
the ambient topology of $M^n$ in the sense that at every generic
point of $\CU$ the two topologies coincide. From
now on, whenever we refer to a topology on $G$, it is $\tau$ we are considering.

\begin{defn} (See \cite{BE}) In an o-minimal structure, a locally definable group $\CU$ is called {\em
connected} if there is no locally definable compatible subset $\emptyset\subsetneqq
\CV\subsetneqq  \CU$ which is both closed and open with respect to the group
topology.
\end{defn}
\begin{remark}\label{connected}
It is easy to see that, in an o-minimal structure, if a locally definable group $\CU$ is
generated by a definably connected set which contains the identity, then it is connected.
\end{remark}

\begin{defn} Given a locally definable group $\CU$ and $\Lambda_0\sub \CU$ a
normal subgroup, we say that $\CU/\Lambda_0$ is \emph{definable} if there
is a definable group $\overline{K}$ and a surjective locally definable
homomorphism $\mu:\CU\to \overline{K}$ whose kernel is $\Lambda_0$.
\end{defn}

We now quote Theorem 3.10 from \cite{ep-defquot} (in a restricted case).

\begin{fact}\label{fromep1}
Let $\CU$ be a connected, abelian locally definable group, which is generated by a definably compact set. Assume that $X\sub \CU$
is a definable set and $\Lam\leqslant \CU$ is a finitely generated subgroup such
that $X+\Lam=\CU$.

Then there is a subgroup $\Lam'\sub \Lam$ such that $\CU/\Lam'$ is a definably compact definable group.
\end{fact}

\subsection{Pushouts and definability}

{\em In the following three subsections, all groups are assumed to be abelian and
all arrows represent group homomorphisms.}

Several steps of the proof require us to construct extensions of abelian groups with
certain maps attached to them. All constructions are standard in the classical
theory of abelian groups but because we are concerned here with definability issues
we review the basic notions (see \cite{Fuchs} for the classical treatment). The
proofs of these basic results are given in the appendix. Although we chose to
present the constructions below in the more common language of pushouts and
pullbacks, it is also possible to carry them out in the less canonical (but possibly
more constructive) language of sections and cocycles.
\begin{defn}
Given homomorphisms
$$\begin{diagram}
\node{A}\arrow{e,t}{\al}\arrow{s,l}{\beta}\node{B}\\
\node{C}
\end{diagram}$$ the triple $(D,\gamma, \delta)$ (or just $D$) is called
{\em a pushout} (of $B$ and $C$ over $A$ via $\alpha,\beta,\gamma,\delta$) if the
following diagram commutes
$$\begin{diagram}
\node{A}\arrow{e,t}{\al}\arrow{s,l}{\beta}\node{B}\arrow{s,r}{\gamma}\\
\node{C}\arrow{e,t}{\delta}\node{D}
\end{diagram}$$
and for every commutative diagram
\begin{equation}\label{eq-push1}\begin{diagram}
\node{A}\arrow{e,t}{\al}\arrow{s,l}{\beta}\node{B}\arrow{s,r}{\gamma'}\\
\node{C}\arrow{e,t}{\delta'}\node{D'}
\end{diagram}\end{equation} there is a unique $\phi:D\to D'$
such that $\phi\gamma=\gamma'$ and $\phi\delta=\delta'$.

If $A,B,C,D$ and the associated maps are (locally) definable, and if for every
(locally) definable $D', \gamma',\delta'$ there is a (locally) definable
$\phi:D\to D'$ as required then we say that {\em the pushout is (locally)
definable}.
\end{defn}

\begin{prop}\label{pushout1}
Assume that we are given the following diagram
$$\begin{diagram}
\node{A}\arrow{e,t}{\al}\arrow{s,l}{\beta}\node{B}\\
\node{C}
\end{diagram}$$

\noi (i) Let $(D, \gamma, \delta)$ be a pushout. Then
$$\ker(\gamma)=\alpha(\ker(\beta)).$$

Moreover, if $\beta$ is surjective, then so is $\gamma$. If $\alpha $ is injective,
then so is $\delta$.

\noindent (ii)  Suppose that all data are definable. Then there exists a definable pushout
$(D,\gamma,\delta)$, which is unique up to definable isomorphism.

\noindent (iii) Suppose that all data are locally definable and $\al(A)$ is a compatible subgroup of $B$. Then there exists a locally definable pushout $(D,\gamma,\delta)$, which is unique up to  locally definable isomorphism.

Assume now that $\al$ is injective. If we let $E=B/\al(A)$ and $\pi:B\to E$ the projection map then there is a locally definable surjection
$\pi':D\to E $ such that the diagram below commutes and both sequences are exact. In
particular, $\ker(\pi')=\delta(C)$ is a compatible subgroup of $D$.
\[
\begin{diagram}
\node{0}\arrow{e}\node{A}\arrow{s,l}{\beta} \arrow{e,t}{\al}\node{B}
\arrow{s,r}{\gamma}\arrow{e,t}{\pi}
\node{E} \arrow{s,r}{id_E}\arrow{e} \node{0}\\
\node{0}\arrow{e}\node{C} \arrow{e,t}{\delta}\node{D} \arrow{e,t}{\pi'}\node{E}
\arrow{e} \node{0}
\end{diagram}
\]

\end{prop}

We also need the following general fact, for which we could not find a reference
(see appendix for proof):
\begin{lemma}\label{pushout2} Assume that we are given the following commutative diagram
\begin{equation}
\begin{diagram}
\node{A}\arrow{e,t}{\alpha}\arrow{s,l}{\beta}\node{B}\arrow{s,l}{\gamma}\\
\node{C}\arrow{e,t}{\delta}\arrow{s,l}{\eta}\node{D}\arrow{s,l}{\mu}\\
\node{E}\arrow{e,t}{\xi}\node{F}
\end{diagram}
\end{equation}
with $D$ the pushout of $B$ and $C$ over $A$ (via $\alpha,\beta,\gamma,\delta$), and
$F$ the pushout of $B$ and $E$ over $A$ (via $\alpha$, $\eta\beta$, $\mu\gamma$ and
$\xi$). Then $F$ is also the pushout of $E$ and $D$ over $C$ (via
$\eta,\delta,\mu,\xi$).
\end{lemma}

\subsection{Pullbacks and definability}

\begin{defn}
Given homomorphisms $$\begin{diagram}
\node{} \node{B}\arrow{s,r}{\al}\\
\node{C}\arrow{e,b}{\beta}\node{A}
\end{diagram}$$
the triple $(D,\gamma,\delta)$ (or just $D$) is called {\em a pullback} (of $B$ and
$C$ over $A$ via $\alpha,\beta,\gamma,\delta$) if the following diagram commutes
$$\begin{diagram}
\node{D}\arrow{e,t}{\gamma}\arrow{s,l}{\delta} \node{B}\arrow{s,r}{\al}\\
\node{C}\arrow{e,b}{\beta}\node{A}
\end{diagram}$$ and  for every commutative diagram
\begin{equation}\label{eq-pull1}\begin{diagram}
\node{D'}\arrow{e,t}{\gamma'}\arrow{s,l}{\delta'} \node{B}\arrow{s,r}{\al}\\
\node{C}\arrow{e,b}{\beta}\node{A}
\end{diagram}\end{equation} there is a unique $\phi:D'\to D$ such that
$\gamma\phi=\gamma'$ and $\delta\phi=\delta'$.

If $A,B,C,D$ and the associated maps are (locally) definable, and if for every
(locally) definable $D', \gamma',\delta'$ there is a (locally) definable
$\phi:D'\to D$ as required then we say that {\em the pullback is (locally)
definable}.
\end{defn}

\begin{prop}\label{pullback1}
Assume that we are given the following  diagram
$$\begin{diagram}
\node{} \node{B}\arrow{s,r}{\al}\\
\node{C}\arrow{e,b}{\beta}\node{A}
\end{diagram}$$

\noi (i) Let $(D, \gamma, \delta)$ be a pullback. Then
$$\gamma(\ker(\delta))=\ker(\alpha).$$

Moreover, if $\beta$ is surjective, then so is $\gamma$. If $\alpha $ is injective,
then so is $\delta$.

\noindent (ii) Suppose that all data are definable. Then there exists a definable pullback
$(D,\gamma,\delta)$, which is unique up to definable isomorphism.

\noindent (iii) Suppose that all data are locally definable. Then there exists a locally definable pullback $(D,\gamma,\delta)$, which is unique up to locally definable isomorphism.

Assume now that $\beta$ is surjective. Let $G=\ker(\gamma)$ and $H=\ker(\beta)$. Then $G$, $H$ are
locally definable and compatible in $D$ and $C$, respectively. Moreover, there is a
locally definable isomorphism $j:G\to H$ such that the following diagram commutes
and both sequences are exact.
$$\begin{diagram}\node{0}\arrow{e}\node{G}\arrow{s,l}{j}\arrow{e,t}{id_G}\node{D}\arrow{e,t}{\gamma}\arrow{s,r}{\delta}\node{B}
\arrow{s,r}{\al}\arrow{e}\node{0}\\
\node{0}\arrow{e}\node{H}\arrow{e,t}{id_H}\node{C}\arrow{e,t}{\beta}\node{A}\arrow{e}\node{0}
\end{diagram}
$$

\end{prop}




\subsection{Additional lemmas}

\begin{lemma}\label {ext3}
Assume that the sequence $$\begin{diagram}
\node{0}\arrow{e}\node{A}\arrow{e,t}{i}\node{B}\arrow{e,t}{\pi}\node{C}\arrow{e}
\node{0}
\end{diagram}
$$ is exact and that we have a surjective homomorphism $\mu:C\to D$. Let $A'=\ker(\mu
\pi)\sub B$. Then  the following diagram commutes and  both sequences are exact. If
all data are locally definable then so is $A'$ and the associated maps.
\[
\begin{diagram}
\node{0}\arrow{e}\node{A}\arrow{s,l}{i} \arrow{e,t}{i}\node{B}
\arrow{s,r}{id}\arrow{e,t}{\pi}\node{C}
 \arrow{s,r}{\mu}\arrow{e} \node{0}\\
\node{0}\arrow{e}\node{A'} \arrow{e,t}{id}\node{B} \arrow{e,t}{\mu\pi}\node{D}
\arrow{e} \node{0}
\end{diagram}
\]

\end{lemma}
\proof This is trivial.\qed

\begin{lemma}\label{ext4} Assume that we have surjective homomorphisms $F:B\to G$ and $F':B\to G'$ with $\ker(F')\sub \ker(F)$. Then there
is a canonical surjective homomorphism $h:G'\to G$, given by $h(g')=g$ if and only
if there exists $b\in B$ with $F'(b)=g'$ and $F(b)=g$. The kernel of $h$ equals
$F'(\ker (F))$ and if all data are locally definable then so is $h$.
\end{lemma} \proof Algebraically, this is just the fact that if $B_1\sub B_2\sub B$ then
there is a canonical homomorphism  $h:B/B_1\to B/B_2$, whose kernel is $B_2/B_1$.

As for definability, assume that $B,G, G'$, and $F,F'$ are $\bigvee$-definable, and
take definable sets $X\sub G$ and $X'\sub G'$. We want to show that the intersection
of $graph(h)$ with $X'\times X$ is definable. Since  $F'$, $F$ are
$\bigvee$-definable and surjective, there exists a definable $Y\sub B$ such that
$F'(Y)\supset X'$ and $F(Y)\supseteq X$. Now, for every $g'\in X'$ there exists
$b\in Y$ such that $F'(b)=g'$, and we have $h(g')=F(b)$. Thus, the intersection of
$graph(h)$ with $X'\times X$ is definable. \qed

\begin{rmk}
 All statements from Proposition \ref{pushout1} to Lemma \ref{ext4} hold under the more general assumption that $\cal{M}$ is any sufficiently saturated structure
  (not necessarily o-minimal) which has strong definable choice. This is because the definability issues
   in the statements are all based on Fact \ref{edmundo}, which can be proved for such a more general $\cal{M}$.
\end{rmk}

\section{Preliminaries II: Semi-bounded sets}\label{sec-sbd}

\subsection{Long cones and long dimension}

In this section we recall some notions from \cite{el-sbd} and prove basic facts that
follow from that paper.

A \emph{$k$-long cone in $M^n$} is a set of the form
$$C=\left\{ b+\sum_{i=1}^k \lambda_i(t_i) :b\in B ,\, t_i \in J_i \right\},$$
where $B$ is a short cell, each $J_i=(0,a_i)$ is a long interval (with $a_i$
possibly $\infty$) and $\lambda_1,\ldots,\lambda_k$ are $M$-independent partial
linear maps from $(-a_i,a_i)$ into $M^n$ (by $M$-independent we mean: for all
$t_1,\dots, t_k \in M$, if $\lambda_1(t_1)+\dots+\lambda_k(t_k)=0$ then $t_1=\cdots
=t_k=0$). It is required further that for each $x\in C$ there are unique $b$ and
$t_i$'s with $x=b+\sum_{i=1}^k \lambda_i (t_i)$ (we refer to this as ``long cones
are normalized"). So $\dim C=\dim B+k$. A \emph{long cone} is a $k$-long cone for
some $k$. By the normality condition, if $C$ is a $k$-long cone of dimension $k$
then $B$ must be a singleton.

The \emph{long dimension} of a definable set $X\sub M^n$, denoted $\lgdim(X)$, is the maximum $k$ such that $X$
 contains a $k$-long cone. This notion coincides with what we defined as long dimension in the Introduction.
  We call $X$ {\em strongly long} if $\lgdim(X)=\dim(X)$.

Note that if $C$ as above is a bounded cone (namely, all $a_i$'s belong to $M$) then
we can take $B'=\{b+(\lambda_1(a_1/2),\ldots, \lambda_k(a_k/2)):b\in B\}$ and  write
$C=B'+\la C\ra $ where
$$\la C\ra=\left \{\sum_{i=1}^k \lambda_i(t_i):t_i\in (-a_i/2,a_i/2) \right
\}.$$ In this paper, we are interested in bounded cones so we replace $B$ with
$B'$ and write $C=B+\la C\ra$.
\\

As is shown in \cite[Section 5]{el-sbd} the notion of short and long intervals gives
rise to a pregeometry based on the following closure operation:

\begin{defn} Let $\CM$ be an o-minimal expansion of an ordered group. Given $A\sub M $ and $a\in M$,
we say that $a$ is in the short closure
of $A$, $a\in scl(A)$, if there exists an $A$-definable short interval containing
$a$ (in particular, $dcl(A)\sub scl(A)$).

 We say that $B\sub M$ is
$scl$-independent over $A$ if for every $b\in B$, we have $b\notin
scl(B\!\cup\!A\setminus \{b\})$. We let $\lgdim(B/A)$ be the cardinality of a maximal
$scl$-independent subset of $B$ over $A$.
\end{defn}
Notice that if $\CM$ expands a real closed field then every set has long dimension
$0$ over $\emptyset$. On the other hand if $\CM$ is a reduct of an ordered vector
space then $scl(-)=dcl(-)$. Thus, this notion is interesting when $\CM$ is
non-linear and yet does not expand a real closed field (namely, non-linear and
semi-bounded).

  As for the usual o-minimal dimension, the notion of long dimension
for definable sets is compatible with the $scl$-pregeometry in the following sense
(see \cite[Corollary 5.10]{el-sbd}):

\begin{fact} If $X$ is an $A$-definable set in a sufficiently saturated o-minimal
expansion of an ordered group then
$$\lgdim(X)=\max\{\lgdim(x/A):x\in X\}.$$
\end{fact}
We say that $a\in X$ is {\em long-generic over $A$} if $\lgdim(a/A)=\lgdim(X)$.
\\

By \cite[Theorem 3.8]{el-sbd}, if $X$ is $A$-definable of long dimension $k$ and $a$
is long generic in $X$ over $A$ then $a$ belongs to an $A$-definable $k$-long cone
in $X$.
\\

We are now ready to prove two facts which will be used later on.

\begin{fact}\label{Fcont}
Let $F:B\times C\to M^l$ be a definable map, where  $B\sub M^m$ is a short set and
$C\sub M^n$ is strongly long (namely $\lgdim(C)=dim(C)$).
 Then there is an open subset $B_1$ of $B$ and a strongly long
$X\sub C$, with $\dim X=\dim C$, such that $F$ is continuous on $B_1\times X$.
\end{fact}
\begin{proof} We may assume that $B,C$ and $F$ are $\emptyset$-definable.
Pick $b$ generic in $B$  and $c$ which is long-generic in $C$ over $b$. Since $B$ is
short we have
$$\lgdim(bc/\emptyset)=\lgdim(c/b)=\lgdim(C)=\lgdim(B\times C).$$ Because
$\dim C=\lgdim C$, $c$ is also generic over $b$ and, hence, we have
$$\dim(bc/\emptyset)=\dim B\times C.$$ That is, $\la
b,c\ra$ is generic in $B\times C$ so there exists a $\emptyset$-definable
relatively open set $Y\sub B\times C$ containing $\la b,c\ra$, on which $F$ is
continuous. In particular, there exists a relatively open neighborhood $B_1\sub B$,
$b\in B_1$, such that $B_1\times \{c\}\sub Y$. We may assume that $B_1$ is given as
the intersection of a short rectangular neighborhood $V_0$ and $B$. By shrinking
$V_0$ if needed, we may assume that the set of parameters $A$ defining $V_0$ is
$scl$-independent from $\la b,c\ra$ (and contains short elements). Hence
$\lgdim(c/Ab)=\lgdim(c/b)$ so $c$ is still long-generic in $C$ over $Ab$. By
genericity, we can find an $Ab$-definable set $X\sub C$ such that $B_1\times X\sub
Y$. Because $c\in X$, the set $X$ must be strongly long of the same (long) dimension
as $C$.
\end{proof}

\begin{fact}\label{shortimage}
Let $h:X\to W$ be a definable map, where $\lgdim X=\dim X>0$ and $W\sub M^m$ is
short. Then there exists a definable set $Y\sub X$, with $\lgdim Y< \lgdim X$ such
that $h$ is
 locally constant on $X\setminus Y$.
\end{fact}
\begin{proof}
Without loss of generality, $X$, $W$ and $h$ are $\emptyset$-definable. Take $x$
long-generic in $X$ and let $w=h(x)$. Because $w\in W$, we have
$\lgdim(w/\emptyset)=0$ and therefore $x$ is still long-generic in $X$ over $w$. It
follows that there is a $w$-definable set $X_0\sub X$, such that for every $x'\in
X_0$, $h(x')=w.$ The set $X$ is strongly long, so $x$ is also generic in $X$ over
$w$. Hence, the set $X_0$ contains a relative neighborhood of $x$ in $X$, so $h$ is
locally constant at $x$. This is true for every long-generic element in $X$ so the
set of points at which $h$ is not locally constant must have smaller long dimension
than that of $X$.
\end{proof}

\subsection{A preliminary result about definably compact groups}\label{localH}

 We assume that $\la G,+\ra $ is a definable
abelian group. Recall that $X\sub G$ is {\em generic} if finitely many group
translates of $X$ cover $G$. Using terminology from \cite{Ot-Pet}, a definable
 set $X\sub G$ is called {\em $G$-linear} if for every $g,h\in X$ there is an
open neighborhood $U$ of $0$ (here and below, we always refer to the group topology
of $G$), such that $(g-X)\cap U=(h-X)\cap U$. Clearly, every open subset of a
definable subgroup of $G$ is a $G$-linear set. More generally, every group translate
of such a set is also $G$-linear. As is shown in \cite{Ot-Pet}, if a $G$-linear
subset contains $0$ then it contains an infinitesimal subgroup of $G$. When the
group $G$ is $\la M^n, +\ra$  a $G$-linear subset is also called {\em affine}. We
call a definable $G$-linear subset $X\sub G$ {\em a local subgroup of $G$} if it is
definably connected and $0\in X$.

The  $G$-linear set $G_0\sub G$ and the $H$-linear set $H_0\sub H$ are {\em
definably isomorphic} if there exists a definable bijection $\phi:G_0\to H_0$ such
that for every $g,h,k\in G_0$, $g-h+k\in G_0$ if and only if
$\phi(g)-\phi(h)+\phi(k)\in H_0$, in which case we have
$\phi(g-h+k)=\phi(g)-\phi(h)+\phi(k)$. \emph{An isomorphism of local subgroups}
$G_0\sub G$ and $H_0\sub H$, is  further required to send $0_G$ to $0_H$. If
$\phi:G_0\to H_0$ is an isomorphism of local subgroups  then for all $g,k\in G_0$, if
$g+k\in G_0$ then $\phi(g)+\phi(h)\in H_0$ and we have $\phi(g+h)=\phi(g)+\phi(h)$.


Our starting point is  Proposition 5.4 from \cite{el-sbd}, which comes out of the
analysis of definable sets in semi-bounded structures. Recall our notation $C=B+\la
C\ra$ from Section \ref{sec-sbd}. Below we use $\oplus$ and $\ominus$ for group
addition and subtraction in $G$ and use $+$ and $-$ for the group operations in
$\CM$.

\begin{fact}\cite[Proposition 5.4]{el-sbd} \label{pantelis} Let $\la G,\oplus \ra $ be a definably compact
abelian
 group of long dimension $k$. Then $G$ contains
a definable, generic, bounded $k$-long cone $C$ on which the group topology of $G$
agrees with the o-minimal topology. Furthermore, for every $a\in C$ there exists an
open neighborhood $V\sub G$ of $a$ such that for all $x,y\in V\cap a+\la C\ra$,
\begin{equation} \label{cone1} x\ominus a\oplus y=x-a+y.\end{equation}
\end{fact}

 Our goal is to prove:
\begin{prop}\label{groupH} Let $\la G,\oplus\ra $ be a definably compact, definably connected abelian group. Then there exists
a definably connected, $k$-dimensional local subgroup $H\sub G$ and a definable
short set $B\sub G$, $\dim(B)=\dim(G)-k$, satisfying:
\begin{enumerate}
\item $\la H,\oplus\ra$ is definably isomorphic, as a local group, to $\la H',+\ra$,
where $H'=(-e_1, e_1)\times \dots\times (-e_k, e_k)\sub M^k$, with each $e_i>0$ tall
in $M$. In particular, $\dim H=lgdim H=k$.

\item The set $B\oplus H=\{b\oplus h:b\in B\,\, h\in H\}$ is generic in $G$.

\end{enumerate}
\end{prop}
\proof We fix a definably connected short set $B$ and a $k$-long cone $C=B+\la C\ra$
as in Fact \ref{pantelis}.

For $b\in B$, let $C_b$ be the fiber $b+\la C\ra$. Note that for every $x\in C_b$,
and a sufficiently small neighborhood $V$ of $x$, we have $V\cap C_b=V\cap x+\la
C\ra$. Note also that each $C_b$ is an affine subset of $\la M^n,+\ra$.
 Thus, condition (\ref{cone1}) implies that each $C_b$,
locally near every $a\in C_b$, is a $G$-linear subset of $G$, and furthermore the
identity map is locally an isomorphism of $\la C_b,+\ra$ and $\la C_b,\oplus\ra$.
Because the affine topology and the group topology agree on $C$ (and because $C$ is
definably connected in $M^n$), each fiber $C_b$ is definably connected with respect
to the group topology. By \cite[Lemma 2.4]{Ot-Pet}, each $C_b$  is therefore a
$G$-linear (not only locally) subset of $G$ and  the identity map is an isomorphism
of the affine set $\la C_b,+\ra$ and the $G$-linear set $\la C_b,\oplus\ra$.

Let us summarize what we have so far: On one hand, the set $C=B+\la C\ra$ is a
generic set in $G$, which can be written as a disjoint union of affine sets
$\bigcup_{b\in B} C_b$. Furthermore, for each $a,b\in B$ the map
$$f_{a,b}(x)=x-a+b$$ is an isomorphism of the affine sets $C_a$ and $C_b$. On the
other hand, each $C_b$ is also a $G$-linear set, and the same maps $f_{a,b}:C_a\to
C_b$ are isomorphisms of $G$-linear sets (because the identity is an isomorphism of
$\la C_a,+\ra$ and $\la C_a,\oplus\ra$).

Our next goal is to show that, for many $a, b$ in $B$, each map $f_{a,b}(x)$ is not
only a translation in the sense of the group $\la M^n,+\ra$ but also a translation
in $\la G,\oplus\ra$.
\\

We define on $B$ the following equivalence relation: $a\sim b$ if there exists $g\in
G$ such that  we have $f_{a,b}(x)=x\oplus g$ for all $x\in C_a$. Note that for every
$a,b,c\in B$, we have $f_{b,d}\circ f_{a,b}=f_{a,d}$, so it is easy to check that
$\sim$ is an equivalence relation.

\begin{claim} There are only finitely many $\sim$-equivalence classes in
$B$.\end{claim}

\proof  Assume towards contradiction that there are infinitely many classes. By
definable choice, we can find an infinite definable set of representatives for
$B/\sim$. We then replace $B$ by a definably connected component of this set,
calling it $B$ again. So, we may assume that any two $a,b\in B$ are in distinct
$\sim$-classes and that $B$ is still infinite and definably connected. We fix some
$a_0\in B$ and consider the map $F:B\times C_{a_0}\to G$, given by
$F(b,x)=f_{a_0,b}(x)$.

 Since $C_{a_0}$ is strongly long, we can find an open subset $B_1\sub
B$ and a strongly long set $X\sub C_{a_0}$, $\dim X=\dim C_{a_0}$,  such that $F$ is
continuous on $B_1\times X$ with respect to the group topology (Fact \ref{Fcont}).
Without loss of generality we can assume that $X$ is definably compact (we first
take a bounded $X$, then shrink it slightly, and take its topological closure).

Let us fix a  $G$-open chart $V\sub G$ containing $0_G$, and a homeomorphism with an
open affine set $\phi:V\to V'\sub M^{\ell}$ ($\ell=dim G$). Without loss of
generality $\phi(0_G)=0\in M^{\ell}$.  By identifying $V$ and $V'$, we may assume
that $V'\sub G$ is an open set with respect to both the affine and the $G$-topology.

By the definable compactness of $X$, for every neighborhood $W\sub M^{\ell}$ of $0$,
there is a neighborhood $B_2\sub B_1$ of $a_0$, such that for all $b',b''\in B_2$
and $x\in X$, we have $F(x,b')\ominus F(x,b'')\in W$. Indeed, if not then there are
definable curves $x(t)\in X$, $b_1(t), b_2(t)\in B_1$, with $b_1(t)$, $b_2(t)$
tending to $b$ and such that for all $t$,
$$F(x(t),b_1(t))\ominus F(x(t),b_2(t))\notin W.$$ Definable compactness of $X$
implies that $x(t)\to x_0\in X$, so by continuity we have $F(x_0,b)\ominus
F(x_0,b)\notin W, $ contradiction.

 We now fix $W\sub M^{\ell}$
a short neighborhood of $0$, and choose $B_2$ accordingly. If we take distinct
$b',b''$ in  $B_2$ then we obtain a map $h:X\to W$, defined by $h(x)=F(x,b')\ominus
F(x,b'')$. Because $X$ is strongly long, and $W$ is short, the map $h$ must be
locally constant outside a subset of $X$ of long dimension smaller than $k$ (Fact \ref{shortimage}). So, we have
an open neighborhood $V''\sub C_{a_0}$ and an element $g\in G$, such that for all
$x\in V''$, $f_{a_0,b'}(x)\ominus f_{a_0,b''}(x)=g$.

We claim that for all $x\in C_{a_0}$, we have  $f_{a_0,b'}(x)\ominus
f_{a_0,b''}(x)=g$.

First take $x\in V''$ and choose any $y,z\in C_{a_0}$ which are sufficiently close
to each other. Since $C_{a_0}$ is a $G$-linear set, $x\ominus y\oplus z$ is still in
$C_{a_0}$ and still in $V''$. So we have
$$f_{a_0,b'}(x\ominus y\oplus z)=f_{a_0,b'}(x)\ominus f_{a_0,b'}(y)\oplus
f_{a_0,b'}(z)$$ and
$$f_{a_0,b''}(x\ominus y\oplus z)=f_{a_0,b''}(x)\ominus f_{a_0,b''}(y)\oplus
f_{a_0,b''}(z).$$

By subtracting the two equations (in $G$), we obtain $$g=g\oplus
(f_{a_0,b'}(z)\ominus f_{a,b''}(z))\ominus (f_{a,b'}(y)\ominus f_{a,b''}(y)),$$ so
$$f_{a_0,b'}(z)\ominus f_{a,b''}(z)=f_{a,b'}(y)\ominus f_{a,b''}(y)$$ for all $y,z \in
C_{a_0}$ which are sufficiently close to each other. This implies that the function
$f_{a_0,b'}\ominus f_{a_0,b''}$ is locally constant on $C_{a_0}$ so by definable
connectedness, it must be constant on $C_{a_0}$. We therefore showed that
$f_{a_0,b'}\ominus f_{a_0,b''}=g$, so in fact $b'\sim b''$ contradicting our
assumption. Thus $\sim$ has only finitely many classes in $B$.\qed\smallskip

We now return to the relation $\sim$ with its finitely many classes  $B_1,\ldots,
B_m$, and consider the partition of $C$ into $\bigcup_{b\in B_i} C_b$, $i=1,\ldots,
m$. Note that for each $i=1,\ldots, m$ and every $b',b''\in B_i$, there exists $g\in
G$ such that $x\mapsto x\oplus g$ is an isomorphism of the $G$-linear sets $C_{b'}$
and $C_{b''}$.

Since $C$ was generic in $G$, one of these sets is also generic in $G$ (here we use
the definable compactness of $G$). So we assume from now on that for every
$b_1,b_2\in B$ there exists an element $g\in G$ such that $C_{b_1}=C_{b_2}\oplus g$.

Fix $b_0\in B$ and for every $b\in B$ choose an element $g(b)$ in $G$  such that
$C_b=C_{b_0}\oplus g(b)$. If we let  $B'=\{g(b)\oplus b_0:b\in B\}$ and
$H=C_{b_0}\ominus b_0$, then $C=B'\oplus H$.

Let's see that $H$ is as required. Indeed, the map $x\mapsto x\oplus b_0$ is an
isomorphism of the local subgroups $\la H,\oplus\ra$ and $\la C_{b_0},\oplus\ra$. As
we already pointed out, the identity map is an isomorphism of $\la C_{b_0},\oplus
\ra$ and $\la C_{b_0},+\ra$. Finally, $y\mapsto y-b_0$ is an isomorphism of the
affine sets $\la C_{b_0},+\ra$ and $\la \la C\ra,+\ra$. The composition of these
maps is an isomorphism of the local groups $\la H,\oplus\ra$ and
$$H'=\left\la \left(-\frac{a_1}{2},\frac{a_1}{2}\right)\times \cdots \times \left(-\frac{a_k}{2},\frac{a_k}{2}\right),+\right\ra$$ (it sends $0_G$ to $0$).
This ends the proof of Proposition \ref{groupH}.\qed

\section{The universal cover of $G$}

\subsection{Proof of Theorem \ref{thm1}}\label{proof-thm1}
We first prove the abelian case. We proceed with the same notation as in the previous seciton.  Namely, $\la
G,\oplus\ra$ is a definably connected, definably compact abelian group, and $H\sub G$ is the definable strongly
long set from Proposition \ref{groupH}.

Let $f': \la H', +\ra\to \la H, \oplus\ra$ be the acclaimed isomorphism of local
groups. We let $\CH=\la H\ra$ be the subgroup of $G$ generated by $H$. Since $H$ is
a local abelian subgroup of $G$ of dimension $k$, $\CH$ is a locally definable
abelian subgroup of $G$ of dimension $k$ (see \cite[Lemma 2.18]{Ot-Pet}). One can
show that the universal cover of $\CH$ is a locally definable subgroup $\widehat\CH$
of $\la M^k, +\ra$. Indeed, let $\widehat\CH =\la H'\ra$ be the subgroup of $\la
M^k, +\ra$ generated by $H'$. Then we can extend $f'$ to a map $f: \widehat\CH\to
\CH$ with, for every $x_1, \dots, x_l\in H'$,
\[
f(x_1+\dots+x_l)=f'(x_1)\oplus f'(x_2)\oplus\cdots \oplus f'(x_l)
\]
is a $\bigvee$-definable covering map for $\CH$. (The fact that $f$ is well-defined
is provided by the same argument as for \cite[Lemma 4.27]{ElSt}). Since $\wH$ is
divisible and torsion-free, it is the universal cover of $\CH$.

We let $\CH'_0$ be the subset of $M^k$ that consists of all short elements (by this
we mean all elements of $M^k$ all of whose coordinates are short). By \cite[Lemma
3.4]{pet-sbd}, $\la \CH'_0, +\ra$ is a subgroup of $\la M^k,+\ra$ and moreover, it
is a subset of $H'$. It follows that  $\CH_0=f(\CH'_0)$ is a subgroup of $\CH$ which
is isomorphic to $\CH_0'$ (note that by \cite{pet-sbd}, $\CH_0$ is a \Vdef set, but
not, in general, a definable one).

{\em From now on, in order to simplify the notation, we will write $+$ for the group
operation of $G$. In few cases we will also use $+$ for the usual operation on
$M^k$, and this will be clear from the context.}

 We  define $\CB=\bigcup_{n\in\bb{N}} B(n)$, where  $B$ is the definable short set
  from Proposition \ref{groupH}, and the notation $B(n)$ is given in Section \ref{notation}. Since each $B(n)$ is a short definable
set, $\CB$ is a short  locally definable subgroup of $G$.

\begin{claim}\label{H+B} $\CH+\CB=G$.\end{claim}
\proof By Proposition \ref{groupH}, the set $H+B$ is a generic subset of $G$ and  is contained
in $\CH+\CB$ (we use here the fact that $B\sub \CB$ since $0\in B$). Since $G$ is
definably connected we have $\CH+\CB=G$.\qed

The following claim is crucial to the rest of the analysis.
\begin{claim}\label{compatible} The  group
 $\CH_0\cap \cal B$ is compatible in $\cal{B}$, so in particular locally definable.
 \end{claim}
  \proof  Let $X\sub \CB$ be a definable set.
The set $\CB$ is a bounded union of short definable sets, so  $X$ is contained in
one of these and must also be short. We prove that, in general, the intersection of
any definable short $X\sub G$ with $\CH_0$ is definable.

  Since $\CH_0\sub H$ we may assume that $X$ is a
subset of $H$. Let us consider $X'=(f')^{-1}(X)\sub M^k$. Because $f'$ is injective
the set $X'$ is a finite union of definably connected short subsets of $M^k$. It is
easy to see that if one of these short sets contains a short element then every
element of it is short. Thus, if one of these components  intersects $\CH_0'$
non-trivially then it must be entirely contained in $\CH_0'$ (since $\CH_0'$ is the
collection of all short elements). Hence, $X'\cap \CH_0'$ is a finite union of
components of $X'$ and therefore definable. Its image under $f'$ is the definable set
$X\cap \CH_0$.\qed
\\

\noi\textbf{Note:} It is not true in general that $\CH\cap \CB$ is a compatible subgroup of $\CB$ (see Example \ref{example1} below).\\

The decomposition of $\widehat G$ is done through a series of steps.\\

 \noi {\bf
Step 1} By Claim \ref{compatible} and Fact \ref{edmundo}, the quotient
$\CK=\CB/(\CH_0\cap \CB)$ is locally definable and hence we obtain the following
short exact sequence of locally definable groups:
\begin{equation}\begin{diagram} \node{0}\arrow{e}
\node{\mathcal H_0\cap \CB} \arrow{e,t}{i_0}
\node{\CB}\arrow{e,t}{\pi_{\CB}} \node{\CK}
\arrow{e}\node{0}\end{diagram}\end{equation}
\begin{claim}\label{claimstep2}
$\dim H+\dim \CK=\dim G$.
\end{claim}
\proof Because $\CH+\CB=G$, we have
$$\dim \CH+\dim \CB-\dim(\CH\cap \CB)=\dim G.$$
Indeed, this is true for definable groups, and can be proved similarly here by
considering a sufficiently small neighborhood of $0$ in  the locally definable
group $\CH\cap \CB$.

But $\CH_0$ is open in $\CH$ and therefore $\dim (\CH_0\cap \CB)=\dim(\CH\cap \CB)$,
so we also have $\dim \CH+\dim \CB-\dim (\CH_0\cap \CB)=\dim G$. Because
$\CK=\CB/(\CH_0\cap \CB)$, we have $\dim \CB-\dim(\CH_0\cap \CB)=\dim \CK$. We can
now conclude $\dim \CH+\dim\CK=\dim G$.\qed\vskip.2cm

\noi {\bf Step 2}. Since $\CH_0\cap \CB$ embeds into $\CH$ and $\CH_0\cap \CB$ is a
compatible subgroup of $\CB$,  we can apply Lemma \ref{pushout1} and obtain a
locally definable group $D$ (the pushout of $\CH$ and $\CB$ over $\CH_0\cap \CB$)
with the following diagram commuting
\begin{equation}\label{Ex-seq2}
\begin{diagram}
\node{0}\arrow{e}\node{\CH_0\cap \CB}\arrow{s,l}{id}
\arrow{e,t}{i_0}\node{\CB}
\arrow{s,r}{\gamma}\arrow{e,t}{\pi_{\CB}}
\node{\CK} \arrow{s,r}{id_{\CK}}\arrow{e} \node{0}\\
\node{0}\arrow{e}\node{\CH} \arrow{e,t}{j}\node{D}
\arrow{e,t}{\pi_D}\node{\CK} \arrow{e} \node{0}
\end{diagram}
\end{equation}
The maps $\gamma$ and $j$ are injective. Note that since $\CH$ and $\CB$ are
subgroups of $G$, we also have a commutative diagram (with all maps being inclusions)
\begin{equation}\begin{diagram} \node{\CH_0\cap \CB}\arrow{s} \arrow{e}\node{\CB}\arrow{s}\\
 \node{\CH}\arrow{e}\node{G}
 \end{diagram}
 \end{equation}
 It follows from the definition of pushouts that there exists a
 locally definable map $\phi:D\to G$ such that $\phi\gamma:\CB\to G$ and $\phi j:\CH\to G$ are the
 inclusion maps. The restriction of $\phi$ to $j(\CH)$ is therefore injective and
 furthermore,
 the set $\phi(D)$ contains $\CH+\CB$ and
 hence, by Claim \ref{H+B}, $\phi$ is surjective on $G$.
\\

\noi {\bf Step 3} Consider now the universal cover $f:\widehat \CH\to \CH$ where
$\widehat \CH$ is identified with an open subgroup of $\la M^k,+\ra$ as before. As
we saw, the group $\widehat \CH$ has a subgroup $\CH_0'$ which is isomorphic via $f$
to $\CH_0$. Hence, there is a locally definable embedding $\beta:\CH_0\cap \CB\to
\widehat \CH$ such that $f\beta=id_{\CH_0\cap \CB}$. Our goal is to use this
embedding in order to interpolate an exact sequence between the two sequences in
(\ref{Ex-seq2}) (see (\ref{Ex-seq3.5}) below).

We let $\widehat D$ be the pushout of $\widehat\CH$ and $\CB$ over $\CH_0\cap \CB$.
Namely, we have
\begin{equation} \begin{diagram}\label{Ex-seq2.5}
\node{0}\arrow{e}\node{\CH_0\cap \CB}\arrow{s,l}{\beta} \arrow{e,t}{i_0}\node{\CB}
\arrow{s,r}{\gamma''}\arrow{e,t}{\pi_{\CB}} \node{\CK}
\arrow{s,r}{id_{\CK}}\arrow{e} \node{0}\\
\node{0}\arrow{e}\node{\widehat\CH}\arrow{e,t}{\widehat \delta}\node{\widehat
D}\arrow{e,t}{\pi_{\widehat D}}\node{\CK}\arrow{e}\node{0}
\end{diagram}
\end{equation}

\noi{\bf Step 4} Next, we consider the diagram
\begin{equation}\label{Ex-seq2.7}\begin{diagram} \node{\CH_0\cap
\CB}\arrow{s,l}{\beta}\arrow{e,t}{i_0}\node{\CB}\arrow{s,r}{\gamma}\\
\node{\widehat \CH}\arrow{e,t}{jf}\node{D}
\end{diagram}
\end{equation}
Since $f\beta=id$, it follows from (\ref{Ex-seq2}) that the above diagram commutes.
Since $\widehat D$ was a pushout, there exists a locally definable
$\gamma':\widehat D\to D$ such that $\gamma'\gamma''=\gamma$ and $\gamma'\widehat
\delta=jf$.

Putting the above together with (\ref{Ex-seq2}) and (\ref{Ex-seq2.5}), we obtain
\begin{equation} \begin{diagram}\label{Ex-seq3.5}
\node{0}\arrow{e}\node{\CH_0\cap \CB}\arrow{s,l}{\beta} \arrow{e,t}{i_0}\node{\CB}
\arrow{s,r}{\gamma''}\arrow{e,t}{\pi_{\CB}} \node{\CK}
\arrow{s,r}{id_{\CK}}\arrow{e} \node{0}\\
\node{0}\arrow{e}\node{\widehat\CH}\arrow{s,l}{f}\arrow{e,t}{\widehat
\delta}\node{\widehat D}\arrow{s,r}{\gamma'}\arrow{e,t}{\pi_{\widehat D}}\node{\CK}
\arrow{s,r}{id_{\CK}}\arrow{e}\node{0}\\
\node{0}\arrow{e}\node{\CH}\arrow{e,t}{j}\node{D}\arrow{e,t}{\pi_D}\node{\CK}\arrow{e}\node{0}
\end{diagram}
\end{equation}
Note that in order to conclude that the above diagram commutes, we still  need to
verify that the bottom right square commutes, namely, $(id_{\CK})\pi_{\widehat
D}=(\pi_D)\gamma'$.

We now apply Lemma \ref{pushout2} and conclude that the group $D$ is the pushout of
$\CH$ and $\wh D$ over $\wH$. As a corollary we conclude, by Lemma \ref{pushout1} (and the
fact that $f$ is surjective),
\begin{equation}\label{conclusions} (i)\, \,  \pi_{\wh D}=(\pi_D)\gamma'\,\,(ii)\,\,
\ker(\gamma')=\wh \delta(\ker f)\,\, (iii) \mbox{ $\gamma'$ is surjective.}
\end{equation}
In particular, (\ref{Ex-seq3.5}) commutes.

 If we now return to the surjective $\phi:D\to G$ and compose it with
$\gamma'$, we obtain a surjection $\phi\gamma':\widehat D\to G$.

Let us summarize what we have so far:
\begin{equation} \begin{diagram}\label{Ex-seq3.7}
\node{0}\arrow{e}\node{\widehat\CH} \arrow{e,t}{\widehat \delta}\node{\widehat
D}\arrow{s,r}{\phi\gamma'}\arrow{e,t}{\pi_{\widehat D}}\node{\CK}
{}\arrow{e}\node{0}\\
\node{}\node{} \node{G}\node{}\node{}
\end{diagram}
\end{equation}


 \noindent{\bf Step 5}
Let  $\mu:\CU\to \CK$  be the universal cover of $\CK$, (see \cite[Theorem
3.11]{edel2} for its existence and its  local definability) and apply the pullback
construction from Proposition \ref{pullback1} to $\CU$, $\CK$ and $\widehat D$.

We obtain a $\bigvee$-definable group $\widehat G$ (the pullback of $\CU$ and
$\widehat D$ over $\CK$), with associated $\bigvee$-definable maps such that the
following sequences are exact and commute (since the kernels of $\pi_{\wh G}$ and
$\pi_{\wh G}$ are isomorphic we identify them both with $\wH$ and assume that the
map between them is the identity). By Proposition \ref{pullback1}, we also have
\begin{equation}\label{kermu}
\pi_{\wh G}(\ker(\eta))=\ker(\mu).
\end{equation}

\begin{equation}\label{Ex-seq3}
\begin{diagram}
\node{0}\arrow{e}\node{\widehat \CH}\arrow{s,l}{id} \arrow{e,t}{i}\node{\widehat G}
\arrow{s,r}{\eta}\arrow{e,t}{\pi_{\widehat G}}
\node{\CU} \arrow{s,r}{\mu}\arrow{e} \node{0}\\
\node{0}\arrow{e}\node{\widehat \CH} \arrow{e,t}{\widehat
\delta}\node{\widehat D} \arrow{e,t}{\pi_{\widehat D}}\node{\CK}
\arrow{e} \node{0}
\end{diagram}
\end{equation}

Because $\mu$ is surjective, so is $\eta$, so we obtain
 a surjective homomorphism $\wh F:=\phi\gamma'\eta:\widehat G\to G$. It can be inferred from what we have so far that $\CH=\wh F(i(\wH))$.


Note that $\dim \wh G=\dim \CU+\dim \wH$ and, since $\CU$ is the universal cover of
$\CK$, $\dim \CU=\dim \CK$. By Claim \ref{claimstep2}, we have $\dim \wh G=\dim G$.
Note also that  $\CU$ and $\widehat \CH$  are divisible (as connected  covers of
divisible groups) and torsion-free and therefore  so is  $\widehat G$. It follows
that $\widehat F:\widehat G\to G$ is isomorphic to the universal cover of $G$.

We therefore obtain
\begin{equation} \begin{diagram}\label{Ex-seq3.8}
\node{0}\arrow{e}\node{\widehat\CH} \arrow{e,t}{i}\node{\widehat G}\arrow{s,r}{\wh
F}\arrow{e,t}{\pi_{\widehat G}}\node{\CU}
{}\arrow{e}\node{0}\\
\node{}\node{} \node{G}\node{}\node{}
\end{diagram}
\end{equation}



This ends the proof of the first part Theorem \ref{thm1} for an abelian definably
connected, definably compact $G$.\\

Assume now that $G$ is an arbitrary definably compact, definably connected group. By
\cite[Corollary 6.4]{HPP2}, the group $G$ is the almost direct product of the
definably connected groups $Z(G)^0$ and $[G,G]$, and $[G,G]$ is a
semisimple group. The group $G$ is then the homomorphic image of the direct sum
$A\oplus S$ with $A$ abelian,  $S$ semi-simple, both definably compact, and the
kernel of this homomorphism is finite. We may therefore assume that $G=A\oplus S$.
By \cite[Theorem 4.4 (ii)]{HPP2}, the group $S$ is definably isomorphic to a
semialgebraic group over a definable real closed field so it must be short. It
follows that $\lgdim(G)=\lgdim(A)$. By the abelian case, we obtain the following for
the universal cover $\widehat A$ of $A$.
$$\begin{diagram}
\node{0}\arrow{e}\node{\wH} \arrow{e}\node{\widehat{A}}\arrow{s,r}{\wh F}\arrow{e}
\node{\CU}\arrow{e}\node{0}\\ \node{}\node{} \node{A}\end{diagram}
$$ Next, we consider $p:\widehat S\to S$ the universal cover of $S$ (note that $\widehat S$ is also
a compact group). By taking the direct product we obtain:
\begin{equation}
\begin{diagram}
\node{0}\arrow{e}\node{\wH} \arrow{e,t}{i}\node{\widehat G=\widehat{A}\oplus
\widehat S}
\arrow{e,t}{\pi}\arrow{s,r}{\hat F \cdot p}\node{\CU \oplus \wh S} \arrow{e} \node{0}\\
\node{}\node{} \node{G=A\oplus S}\node{}\node{}
\end{diagram}
\end{equation}

In order to finish the proof of Theorem \ref{thm1} we need to see:

\begin{lemma}\label{maxVdef} The group $\CH=\wh F(i(\wH))$ contains every connected, \Vdef
 strongly long subgroup of $G$.
\end{lemma}
\begin{proof} We first prove the analogous result for the universal cover $\wh G$ of
$G$, namely we prove that $i(\wH)$ contains every connected, locally definable,
strongly long subgroup of $\wh G$.  For simplicity, we assume that $\wH\sub \wh G$.

Assume that $\CV\sub \widehat G$ is a connected, \Vdef subgroup with
$\dim(\CV)=\lgdim(\CV)=\ell$. Because $\lgdim(\wh G)=k$ we must have $\ell\leq k$.
We will show that the group $\CV\cap \wH$ has bounded index in $\CV$, so by
connectedness the two must be equal.

Consider $\CU$ from Theorem \ref{thm1}. Because $\CU$ is short, there exists at
least one $u\in \CU$ such that $\lgdim(\pi^{-1}(u)\cap \CV)=\ell$ (see \cite[Lemma
4.2]{el-sbd}). Since $\CV$ is a group we can use translation in $\CV$ to show that
{\em for every} $u\in \pi(\CV)$, we must have $\lgdim(\pi^{-1}(u)\cap \CV)=\ell$. In
particular, $\lgdim(\wH\cap \CV)=\lgdim(\pi^{-1}(0)\cap \CV)=\ell$.

Write $\CV=\bigcup_i V_i$ a bounded union of definable sets which we may assume to be
all strongly long of dimension $\ell$. For every $V_i$, consider the definable
projection $\pi(V_i)\sub \CU$. By Lemma \ref{def-long} (proved in the appendix), the
set $F_i$ of all $u\in \pi(V_i)$ such that $\lgdim(\pi^{-1}(u)\cap V_i)=\ell$ is
definable, so because $\dim(V_i)=\ell$, this set must be finite.

Let $F=\bigcup_i F_i\sub \pi(\CV)$. We claim that $F=\pi(\CV)$. Indeed, if
  $u\in \pi(\CV)\setminus F$ then by the definition of the $F_i$'s,
$\lgdim(\pi^{-1}(u)\cap V_i)<\ell$ for all $i$, which implies that
$\lgdim(\pi^{-1}(u)\cap \CV)<\ell$. This is impossible by our above observation, so
we must have $F=\pi(\CV)$.

Because $F$ is a bounded union of finite sets it follows that the index of
$\CV\cap \wH$ in $\CV$ is bounded. Since $\CV$ is connected it follows that
$\CV\cap \wH=\CV$, so $\CV\sub \wH$.

Assume now that $\CV\sub G$ is a connected, locally definable, strongly long
subgroup of $G$ and let $\wh \CV\sub \wh G$ be the pre-image of $\CV$ under $\wh F$.
The group $\wh \CV$ is strongly long and locally definable, and the connected
component of the identity (see \cite[Proposition 1]{BE}), call it $\wh \CV^0$, is
still strongly long (since it has the same dimension and long dimension as $\wh
\CV$). By what we just saw, $\wh \CV^0$ is contained in $\wH$ and hence $\wh
F(\CV^0)$ is a \Vdef subgroup of $\CH\cap \CV$, which has bounded index in $\CV$.
Because $\CV$ is connected it follows $\wh F(\CV^0)=\CV\sub \CH$.
\end{proof}

This ends the proof of Theorem \ref{thm1}.

\section{Replacing the locally definable group $\CU$ with a definable group}\label{barK}
We now proceed to prove Theorem \ref{thm2}. We first
assume again that $G$ is abelian. The goal is to replace the locally
definable  group $\CU$ in  (\ref{Ex-seq3.8})  with {\em a definable} short group.  We refer to the notation of (\ref{Ex-seq3}) and (\ref{Ex-seq3.8}).\\

 \noi {\bf Step 1} Let
$\Lambda=\ker(\widehat F)$ and let $\Lambda_1=\pi_{\widehat G}(\Lambda)\sub \CU$.

\begin{claim}
The universal cover $\CU$ of $\CK$ from (\ref{Ex-seq3}), together with $\Lam_1$,
satisfy the assumptions of Fact \ref{fromep1}. Namely, $\CU$ is connected,
generated by a definably compact set and there is a definable set $X\sub \CU$ such
that $X+\Lambda_1=\CU$. Moreover, $\Lambda_1$ is finitely generated.
\end{claim}
\begin{proof} The group  $\wh G$ is the universal cover of $G$. We first find a definable, definably connected,
definably compact $X\sub \wh G$ which contains the identity, such that $\wh F(X)=G$.
We start with a definable $X\sub \wh G$ such that $\wh F(X)=G$ and then replace it
with $Cl(X)$. We claim that $Cl(X)$ is definably compact. Indeed, if not then by
\cite[Lemma 5.1 and Theorem 5.2]{ed1}, $\wh G$ has a definable, 1-dimensional
subgroup $G_0$ which is not definably compact. Because $\mu$ is locally definable,
its restriction to $G_0$ is definable so $ker(\wh F)\cap G_0$ is finite and
therefore trivial. Hence $\wh F(G_0)$ is a definable subgroup of $G$ that is not
definably compact, contradicting the fact that $G$ is definably compact. Thus, we
can find a definably compact $X'$ with $X'+\ker(\widehat F)=\wh G$. By \cite[Fact
2.3(2)]{ep-defquot}, $X'$ generates $\wh G$.

 By \cite[Claim 3.8]{edel2}, $\wh G$ is path connected so we can easily replace $X'$ by $X_1\supseteq
X'$ which is definably compact and path connected (connect any two definably
connected components of $X'$ by a definable path). To simplify we call this new set
$X$ again.

Also, by \cite[Theorem 1.4 and Corollary 1.5]{edel2}, $\ker(\widehat F)$ is
isomorphic to the fundamental group of $G$, $ \pi^{def}_1(G)$, which is finitely
generated. It follows that $\Lambda_1$ is finitely generated, $\CU=\pi_{\widehat
G}(X)+\Lambda_1$, and $\pi_{\widehat G}(X)$ is definably compact and definably
connected. Since $X$ generates $\wh G$, the set $\pi_{\widehat G}(X)$ generates
$\CU$. By Remark \ref{connected}, $\CU$ is connected.
\end{proof}

We can now apply Fact \ref{fromep1} and conclude that there is a definably compact
group $\overline{K}$ and a $\bigvee$-definable surjection $\widehat \mu:\mathcal U
\to \overline{K}$ with $\ker(\widehat \mu)=\Lambda_0\sub \Lambda_1.$


 Our  goal is to prove: {\em There are
 locally definable extensions $\overline{G}$ and $G'$ of $\overline{K}$, by the group $\wH$ and $\CH$, respectively,  and surjective
homomorphisms from $\overline{G}$ and $G'$ onto $G$}.

First, by Lemma \ref{ext3}, we have a locally definable group $\widehat
\CH'=\ker(\widehat \mu\pi_{\wh G})=\pi_{\wh G}^{-1}(\Lambda_0)\sub \wh G$ such that (we write $i$
for the identity on $\widehat \CH$ on the top left) the diagram commutes and the
following sequences are exact.

\begin{equation}\label{Ex-seq5}
\begin{diagram}
\node{0}\arrow{e}\node{\widehat \CH}\arrow{s,l}{i} \arrow{e,t}{i}\node{\widehat
G}\arrow{s,r}{id}\arrow{e,t}{\pi_{\wh G}}\node{\mathcal U}
 \arrow{s,r}{\widehat \mu}\arrow{e} \node{0}\\
\node{0}\arrow{e}\node{\widehat \CH'} \arrow{e,t}{id}\node{\widehat G}
\arrow{e,t}{\widehat \mu\pi_{\wh G}}\node{\overline{K}} \arrow{e} \node{0}
\end{diagram}
\end{equation}

Because $ker (\wh \mu)\sub \pi_{\hat G}(\Lambda)$, the group $\wH'$ is contained in
the group $i(\wH)+\Lambda$. Since $i(\wH)$ is a divisible subgroup of $\wH'$, there
exists a subgroup $\Lambda'\sub \Lambda$ such that $\wH'$ equals the direct sum of
$i(\wH)$ and $\Lambda'$. Because $\ker(\pi_{\wh G})=i(\wH)$, the group $\Lambda'$ is
isomorphic, via $\pi_{\wh G}$, to $\Lambda_0$, so $\Lambda'$ is finitely generated. We now
have a group homomorphism $p:\wH'\to \wH$, given via the identification of $\wH'$
with $i(\widehat \CH) \oplus \Lambda'$. Namely, $p(i(h)+\lambda)=h$.


We claim that $p$ is a locally definable map. Indeed, $\wH'$ is the union of sets of the form
$i(H_i)+F_i$, where $H_i$ is definable and $F_i$ is a finite subset of $\Lambda'$.
Because the sum of $\wH$ and $\Lambda'$ is direct, each element $g$ of $i(H_i)+F_i$ has
a unique representation as $g=i(h)+f$, for $h\in H_i$ and $f\in F_i$. Therefore the restriction of $p$ to $i(H_i)+F_i$ is definable. It follows
that $p$ is  locally definable.
\\

\noi{\bf Step 2. } We apply Proposition \ref{pushout1}
to the diagram $$\begin{diagram}\node{\widehat \CH'}\arrow{e,t}{id}\arrow{s,l}{p}\node{\widehat G}\\
\node{\wH}
\end{diagram}$$ and obtain a locally definable pushout $\overline{G}$,  such that the following diagram commutes and the sequences are exact:
\[
\begin{diagram}
\node{0}\arrow{e}\node{\widehat \CH'}\arrow{s,l}{p} \arrow{e,t}{id}\node{\widehat
G}\arrow{s,r}{\widehat \alpha}\arrow{e,t}{\hat\mu\pi_{\widehat G}}\node{\overline{K}}
 \arrow{s,r}{id}\arrow{e} \node{0}\\
\node{0}\arrow{e}\node{\wH} \arrow{e,t}{i_1}\node{\overline{G}}
\arrow{e,t}{\pi_{\overline{G}}}\node{\overline{K}} \arrow{e} \node{0}
\end{diagram}
\]

Because $p$ is surjective the map $\wh \alpha:\wh G\to \overline{G}$ is also
surjective. Moreover, by Lemma \ref{pushout1}, the kernel of $\wh \alpha$ equals $\ker
p=\Lambda'$ so is contained in $\Lambda=\ker(\wh F)$.
\\

\noi{\bf Step 3. } We now have surjective maps $\widehat F:\widehat G\to G$ and $\wh
\alpha:\widehat G\to \overline{G}$, both $\bigvee$-definable with $\ker(\wh
\alpha)\sub \ker(\widehat F)$. By Lemma \ref{ext4} we have a $\bigvee$-definable surjective
$\overline{F}:\overline{G}\to G$, with $\ker(\overline{F})= \wh \alpha(\ker (\wh F))$.
We therefore obtained the following  diagram:

\begin{equation}\label{abeliancase3}
\begin{diagram}
\node{0}\arrow{e}\node{\wH} \arrow{e,t}{i_1}\node{\overline{G}}
\arrow{e,t}{\pi_{\overline{G}}}\arrow{s,r}{\overline{F}}\node{\overline{K}} \arrow{e} \node{0}\\
\node{}\node{} \node{G}\node{}\node{}
\end{diagram}
\end{equation}

Finally, let us calculate $\ker(\overline{F})$: Recall that $\Lambda'$ is isomorphic
to $\Lambda_0$ the kernel of the universal covering map $\wh \mu:\CU\to
\overline{K}$. Because $\overline{K}$ is a short definably compact group, it follows from \cite{EO}
that $\ker(\wh \mu)=\pi^{def}_1(\overline{K})=\Z^d$, where
$\pi^{def}_1(\overline{K})$ is the o-minimal fundamental group of $\overline{K}$ and
$$d=\dim(\overline{K})=\dim(\CU)=\dim(G)-k,$$ for $k=\lgdim(G)$. The map $\wh F:\wh
G\to G$ is the universal covering map of $G$ and therefore, as shown in
\cite[Theorem 1.4, Corollary 1.5]{edel2}, $\ker(\wh F)=\pi^{def}_1(G)=\Z^{\ell}$, for
some $\ell$. Furthermore, for every $m\in \N$, the group of $m$-torsion points
$G[m]$ is isomorphic to $(\Z/m\Z)^{\ell}$. By \cite[Theorem 7.6]{pet-sbd},
$G[m]=(\Z/m\Z)^{\dim(G)}$, hence we can conclude $$\Lambda=\ker(\wh
F)=\pi_1^{def}(G)=\Z^{\dim G}.$$

We now have $\ker (\overline{F})=\wh \alpha(\Lambda)\simeq \Lambda/\Lambda'$, with
$\Lambda\simeq \Z^{\dim(G)}$ and $\Lambda'\simeq \Z^{\dim(G)-k}$. Hence,
$\ker(\overline{F})$ is isomorphic to the direct sum of $\Z^{k}$ and a finite group,
as required.
\\

\noi{\bf Question } Can $\overline K$ be chosen so that $\ker(\overline{F})\simeq \Z^k$?\\
\\

Next, consider  $\CH\sub G$ as in Theorem \ref{thm1}. We want to see that we can
obtain a similar diagram to  (\ref{abeliancase3}), with $\CH$ instead of $\wH$. For
simplicity, assume that $i_1$ is the identity. First notice that by the last clause
of Theorem \ref{thm1}, we must have $\overline{F}(\wH)\sub \CH$. However, using
exactly the same proof as in Lemma \ref{maxVdef}, we can show that that
$\overline{F}(\wH)$ is also the largest connected  strongly long, locally definable,
subgroup of $G$, hence it equals $\CH$. We therefore have
$$\begin{diagram}
\node{\wH}\arrow{e,t}{i_1}\arrow{s,l}{\overline{F}\upharpoonright\wH}\node{\overline{G}}\\
\node{\CH}
\end{diagram}
$$

We can now obtain $G'$, the pushout of $\overline{G}$ and $\CH$ over $\wH$:
\[
\begin{diagram}
\node{0}\arrow{e}\node{\wH}\arrow{s,l}{\overline{F}\upharpoonright\wH} \arrow{e,t}{
i_1}\node{\overline{G}}\arrow{s,r}{
\alpha'}\arrow{e,t}{\pi_{\overline{G}}}\node{\overline{K}}
 \arrow{s,r}{id}\arrow{e} \node{0}\\
\node{0}\arrow{e}\node{\CH} \arrow{e,t}{i'}\node{{G'}}
\arrow{e,t}{\pi_{G'}}\node{\overline{K}} \arrow{e} \node{0}
\end{diagram}
\]

Clearly,  $\ker(\overline{F}\upharpoonright \wH)\sub \ker(\overline{F})$, so by
Proposition \ref{pushout1}, $\ker(\alpha')=i(ker \overline{F}\upharpoonright H)\sub
\ker \overline{F}$. By Lemma \ref{ext4}, we have a homomorphism from $G'$ onto $G$
as we want. We therefore have:
\begin{equation}\label{abeliancase4}
\begin{diagram}
\node{0}\arrow{e}\node{\CH} \arrow{e,t}{i'}\node{G'}
\arrow{e,t}{\pi_{\overline{G}}}\arrow{s,r}{h'}\node{\overline{K}} \arrow{e} \node{0}\\
\node{}\node{} \node{G}\node{}\node{}
\end{diagram}
\end{equation}
This ends the switch from (\ref{abeliancase3}) to (\ref{abeliancase4}), and with
that the proof of Theorem \ref{thm2} in the case that $G$ is abelian. In order to
conclude the same result for arbitrary definably compact, definably connected $G$,
we repeat the same arguments as in the last part of the proof of Theorem
\ref{thm1}.\qed\vskip.2cm

\subsection{Special cases}\label{somexample}
As was pointed out earlier, we use Fact \ref{fromep1} to guarantee that there is a
definable group $\overline{K}$ and a $\bigvee$-definable surjection $\widehat
\mu:\mathcal U \to \overline{K}$ with $\Lambda_0:=\ker(\wh \mu)$ a subgroup of
$\pi_{\wh G}(\ker \wh F)$ (see notation of Theorem \ref{thm1}). In certain simple
cases we can see directly why such $\Lambda_0$ exists, without referring to
Fact \ref{fromep1}:

Assume $G$ is abelian. Let $\CK$ and $\CH$ be as in Section \ref{proof-thm1}.
Namely, $\CK$ is the group obtained as the quotient of the locally definable
subgroup $\CB$ of $G$ by the compatible subgroup $\CH_0\cap \CB$, and $\CH$ is the
largest locally definable, connected strongly long subgroup of $G$.
\\

\noindent (1) {\bf Assume that $\CK$ is definable}.
\vskip.2cm

In this case we take $\Lam_0=\ker(\mu)$, where $\mu:\CU\to \CK$. Obviously,
$\CU/\Lambda_0$ is definable, so we need only to see that $\Lam_0\sub \pi_{\wh
G}(\ker \wh F)$. Let $u\in \ker(\mu)$. By (\ref{kermu}), $u=\pi_{\wh G}(v)$, for
some $v\in \ker(\eta)$. But then $\wh F(v)=\phi \gamma' \eta(v)=0$,so $U\in \pi_{\widehat G}(\ker \widehat F).$\\

\noi (2) {\bf Assume that $\CH$ is definable}.
\vskip.2cm

We denote by $\overline{K}$ the definable group $G/\CH$. From Theorem \ref{thm1} and
its proof we obtain the following commutative diagram.
$$\begin{diagram}
\node{0}\arrow{e}\node{\wH}\arrow{e,t}{i}\arrow{s,l}{f}\node{\wh G}\arrow{s,l}{\wh
F}\arrow{e,t}{\pi_{\wh G}}\node{\CU}\arrow{e}\node{0} \\
\node{0}\arrow{e}\node{\CH}\arrow{e,t}{id}\node{G}\arrow{e,t}{\pi_G}\node{\overline{K}}\arrow{e}\node{0}
\end{diagram}
$$
But now there is a unique map $\mu:\CU\to \overline{K}$ which makes the above
diagram commute, and it is easy to verify by construction that $\ker(\mu)\sub
\pi_{\wh G}(\ker(\wh F))$. We now take $\Lambda_0=\ker(\mu)$.
\\

\section{Examples}\label{nonextension}

In this section we provide examples that motivate the statements of Theorem
\ref{thm1} and \ref{thm2}.
 More specifically, we give examples of definably compact groups  which cannot themselves be written as extensions of short
 (locally) definable groups by strongly long (locally) definable subgroups. This is what forces us to move our analysis to the
 level of universal covers.

In the following examples, we fix $\cal{M}=\langle M, +, <, 0, R\rangle$ to be an
expansion of an ordered divisible abelian group by a real closed field $R$, whose
domain is a bounded interval $(0, a)\sub M$. In particular, $\cal M$ is
semi-bounded, o-minimal, and $(0, a)$ is short. Let also $b\in M$ be any tall
positive element. In the first two examples, we define semi-linear groups which have
the same domain $[0, a)\times [0,  b)$ but different operations.

\begin{example}\label{example1} Pick any $0< v_1< a$ such that $a$ and $v_1$ are $\bb{Z}$-independent. Let $L$ be
 the subgroup of $\la M^2, +\ra$ generated by the vectors $\la a, 0\ra$ and $\la v_1, b\ra$, and
 let $G=\la [0, a)\times [0,  b), \star, 0\ra$ be the group with
\[
x\star y = z\, \Lrarr\, x+y-z\in L.
\]
By \cite[Claim 2.7(ii)]{ElSt}, $G$ is definable.

Let us see what the various groups of Theorems \ref{thm1} and \ref{thm2} are in this
case.

We let $\wh G$ be the subgroup of $M^2$ generated by $[0,a]\times [0,b]$. The group
$\wh G$ is torsion-free and it is easy to see that there is a locally definable
covering map $\wh F:\wh G\to G$. Hence, $\wh G$ is the universal cover of $G$. The
group $\wH=\{0\}\times \bigcup_n(-nb,nb)$,  is a locally definable compatible
subgroup of $\wh G$ and the quotient $\wh G/\wH$ is isomorphic to the short group
$\bigcup_n(-na,na)$.

We have $\lgdim(\wH)=\dim(\wH)=1$, so $\wH$ is strongly long. As in the proof of
Proposition \ref{maxVdef}, the group $\wH$ is the largest strongly long, connected,
locally definable subgroup of $\wH$.

Now, we let $\CH=\wh F(\wh \CH)$.  This is  the subgroup of $G$ generated by the set
$H=\{0\}\times [0, b)$ and we can describe it explicitly. Let $S\sub [0,a)$ be the
set containing all elements of the form $n(a-v_1) \mod a$. By the choice of $v_1$,
the set $S$ has to be infinite. By the definition of the operation $\star$, it is
easy to see that
\[
\CH=\bigcup_{s\in S}  \{s\}\times [0, b),
\]
which is not definable (so in particular not compatible in $G$). This shows the need
in Theorem \ref{thm1} to work with the universal cover of $G$ rather than with $G$
itself. Note that $\wh F$ restricted to $\wh H$ is an isomorphism onto $\CH$.

In fact, $G$ does not contain any infinite strongly long  definable subgroup.
Indeed, if it did, then its connected component should be contained in $\CH$ and
therefore the pre-image of this component under $\wh F\upharpoonright \wH$ would be a
proper definable subgroup of $\wH$ and, thus, of $\la M, +\ra$, a contradiction.

Now consider the subgroup $K=\la [0, a)\times \{0\}, \star, 0\ra$ of $G$ and let
$\wh K$ be its universal cover. We can write
$$G=\CH\star K.$$
Of course $\CH\cap K$ is infinite, so this is not a direct sum. However, the
universal cover $\wh G$ of $G$ is a direct sum
$$\wh G=\wH\oplus \wh K,$$
whereas, if we let
$$\overline{G}= \wH \oplus K,$$
then we can define a surjective homomorphism $\overline{F}:\overline{G}\to G$ with
$\ker \overline{F}\simeq \Z (0, b)$.

We finally observe in this example that $\CH \cap K=S$ is not a compatible subgroup of $K$, which indicates the need for passing to $\CH_0$ in the proof of Theorem \ref{thm1} (see Claim \ref{compatible}).
\end{example}

\begin{example}
Pick any $0<u_2<b$ such that  $u_2$ and $b$ are $\bb{Z}$-independent. Let $L$ be the
subgroup of $\la M^2, +\ra$ which is generated by the two vectors $\la a, u_2\ra$
and $\la 0,b\ra$, and let again $G=\la [0, a)\times [0,  b), \star, 0\ra$ be the
group with
\[
x\star y = z\, \Lrarr\, x+y-z\in L.
\]
Here we observe that $H=\{0\}\times [0, b)$ itself is the largest strongly long locally
definable subgroup of $G$ and, hence, $G$ is itself an extension of a short
definable group by $H$. However, $H$ does not have a definable complement in $G$;
namely, $G$ cannot be written as a direct sum of $H$ with some definable subgroup of
it. The proof of this goes back to \cite{str}. See also \cite{PetSte}.

The universal cover $\wH$ of $H$ is again the subgroup of $M^2$ generated by $H$.
Let $\CK$ be the subgroup of $G$ generated by $K=[0, a)\times \{0\}$, and $\wh K$
its universal cover. Then we can write
$$G=H \star \CK, $$
 where again $H\cap \CK$ is not finite, so this is not a direct sum.  The universal cover $\wh G$ of $G$ is again a direct sum
$$\wh G= \wH \oplus \wh K.$$
If we let $\overline{K}=\la [0, a)\times \{0\}, \star_K, 0\ra$ be the group with
operation $\star_K =+\mod a$, then we can define a \emph{suitable} extension $\overline{G}$ of $\overline{K}$ by $\wH$
\[
\begin{diagram}
\node{0}\arrow{e}\node{\wH} \arrow{e}\node{\overline{G}}
\arrow{e}\node{\overline{K}} \arrow{e} \node{0}
\end{diagram}
\]
and a surjective homomorphism $\overline{F}:\overline{G}\to G$ with
$\ker \overline{F}\simeq \Z (0, b)$.
\end{example}

We finally give an example for Theorems \ref{thm1} and \ref{thm2} of a definable
group $G$ which contains no infinite proper definable subgroup.

\begin{example} Pick any $0< v_1< a$ such that $a$ and $v_1$ are $\bb{Z}$-independent, and any $0<u_2<b$
 such that  $u_2$ and $b$ are $\bb{Z}$-independent. Let $L$ be the subgroup of $\la M^2, +\ra$
 which is generated by the vectors $\la a, u_2\ra$ and $\la v_1, b\ra$. We define the group $G$ with domain
\[
\big([0, a)\times [0, b- u_2)\big) \cup \big([v_1, a)\times [b-u_2, b)\big),
\]
and group operation again
\[
x\star y = z\, \Lrarr\, x+y-z\in L.
\]
It is not too hard to verify that the above is indeed a definable  group - this will
appear in a subsequent paper (\cite{el-affine}).

In this case, $G$ does not contain any infinite proper definable subgroup. This
again originates in \cite{str}. We let $\CH$ the subgroup of $G$ generated by
$H=\{0\}\times [0, b- u_2)$, and $\wH$ its universal cover. We also let  $\CK$ be
the subgroup of $G$ generated by $K=[0, a)\times \{0\}$, and $\wh K$ its universal
cover. Then we have:
 $$G=\CH \star \CK,$$
 with $\CH\cap \CK$ infinite, and
$$\wh G= \wH \oplus \wh K.$$
Finally, if we let $\overline{K}=\la [0, a)\times \{0\}, \star_K, 0\ra$ be the group
with operation $\star_K =+\mod a$, then we can define a suitable extension $\overline{G}$ of $\overline{K}$ by $\wH$
\[
\begin{diagram}
\node{0}\arrow{e}\node{\wH} \arrow{e}\node{\overline{G}}
\arrow{e}\node{\overline{K}} \arrow{e} \node{0}
\end{diagram}
\]
and a surjective homomorphism $\overline{F}:\overline{G}\to G$ with
$\ker \overline{F}\simeq \Z (v_1, b)$.
\end{example}

\section{Compact Domination}
\label{CD-section} Let us first recall (\cite[Section 7]{HPP}) that for a definable,
or \Vdef group $\CU$, we write $\CU^{00}$ for the smallest, if such exists,
type-definable subgroup of $\CU$ of bounded index (in particular we require that
$\CU^{00}$ is contained in a definable subset of $\CU$). Note that a type-definable
subgroup $\CH$ of $\CU$ has bounded index if  and only if  there are no new
cosets of $\CH$ in $\CU$ in elementary extensions of $\CM$. A definable $X\sub \CU$
is called \emph{generic} if boundedly many translates of $X$ cover $\CU$. In
\cite[Theorems 2.9 and 3.9]{ep-defquot} we established conditions so that $\CU^{00}$
and generic sets exist.

Let $G$ be a
 definably connected, definably compact, abelian definable group and $\pi:G\to G/G^{00}$  the natural
projection. We equip the compact Lie group $G/G^{00}$ with the Haar measure, denoted
by $m(Z)$, and prove: for every definable $X\sub G$, the set of $h\in G/G^{00}$ for
which $\pi^{-1}(h)\cap X\neq \emptyset $ and $\pi^{-1}(h)\cap (G\setminus X) \neq
\emptyset$ has measure zero. As is pointed out in \cite{HPP}, it is sufficient to
prove that
\begin{equation}\label{eqn-defcomp}
\text{for every definable $X\sub G$, if $\dim X<\dim G$, then $m(\pi X)=0$.}
\end{equation}
We say then that $G$ (and $\pi$) \emph{satisfy Compact Domination}. When $G$ is
locally definable and $G^{00}$ exists then $G/G^{00}$ is locally compact (see
\cite[Lemma 7.5]{HPP}) and so admits Haar measure as well.  We still say that $G$
satisfies compact domination if (\ref{eqn-defcomp}) holds.

We split the argument into two cases:\vskip.2cm

\noindent {\bf  I. $G$ is abelian.}\vskip.2cm

Consider the universal covering map $\phi:\wh G\to G$ and the commutative diagram in
\cite[Proposition 3.8]{ep-defquot}
\begin{equation}\begin{diagram}
\node{\widehat G}\arrow{s,l}{\pi_{\widehat G}}\arrow{e,t}{\phi}\node{G}\arrow{s,r}{\pi_G}\\
\node{\wh G/\wh G^{00}}\arrow{e,t}{\phi'}\node{G/G^{00}}
\end{diagram}.\end{equation}
Using the fact that $\ker \phi$ has dimension zero and $ker \phi'$ is countable, it
is not hard to see that $G$ satisfies Compact Domination if and only if $\wh G$
does. Our goal is then to prove (\ref{eqn-defcomp}) for the universal cover $\wh G$.

Recall by Theorem \ref{thm1} the sequence:
$$
\begin{diagram}
\node{0}\arrow{e}\node{\widehat\CH} \arrow{e,t}{i}\node{\widehat G}
\arrow{e,t}{f}\node{\CU} \arrow{e} \node{0}
\end{diagram}
$$
with $\wH$ an open subgroup of $\la M^k,+\ra $, $\lgdim(\wH)=k=\lgdim(\wh G)$ and
$\CU$ a short \Vdef group of dimension $n$. Note that $\wh G$ contains a definable
generic set (any definable set which projects onto $G$), and hence so does $\CU$. By
\cite[Theorem 3.9]{ep-defquot}, $\CU$ has a definable, definably compact quotient
$K$, and the homomorphism from $\CU$ onto $K$ has kernel of dimension zero. By
\cite{HP}, the group $K$, with its map onto $K/K^{00}$ satisfies Compact Domination,
and therefore $\pi_{\CU}:\CU\to \CU/\CU^{00}$ also satisfies Compact Domination.

We now consider $\wH$ and first claim:
\begin{equation}\label{shortinH}
\text{$\wH^{00}$ exists and contains the set of all short elements in $M^k$.}
\end{equation}
Indeed, recall from  Section \ref{localH} that $\wH$ is generated by a subset $H'\sub M^k$,
 $$H'=(-e_1, e_1)\times \dots\times (-e_k, e_k),$$ with each
$e_i>0$ tall in $M$. We define, for each $n\in \bb{N}$, $H_i= \frac{1}{n} H'$, and
claim that
$$\wH^{00}=\bigcap_n H_n.$$
Indeed, $\bigcap_n H_n$ is a torsion-free subgroup of $\wH$. Moreover, each $H_n$ is
generic in $\wH$ because  we have $\wH=H_n+\Z e_1+\cdots +\Z e_k$. It follows that
$\bigcap_n H_n$ has bounded index in $\wH$,
 and thus \cite[Proposition 3.6]{ep-defquot} gives $\wH^{00}=\bigcap_n H_n.$
 Finally, since each $e_i$ is tall, it is easy to verify that each short tuple in $M^k$ must be contained in
 $\bigcap_n H_n$.

We now claim that $\wh G^{00}\cap i(\wH)=i(\wH^{00})$. This follows from the fact
that $\wh G^{00}\cap i(\wH)$ has bounded index in $i(\wH)$ and it is torsion-free
(\cite[Proposition 3.6]{ep-defquot}). Next, we claim that $f(\wh G^{00})=\CU^{00}$.
Since $f(\wh G^{00})$ has bounded index it must contain $\CU^{00}$. Because $\wh
G^{00}$ is torsion-free and $\ker(f)=i(\wH^{00})=i(\wH)\cap \wh G^{00}$ is divisible
(\cite[Proposition 3.5]{ep-defquot}), it follows that $f(\wh G^{00})$ is torsion-free so must
equal $\CU^{00}$. We therefore have the following commutative diagram of exact
sequences:

\begin{equation}\label{CDdiag}
\begin{diagram}
\node{0}\arrow{e}\node{\wH}\arrow{e,t}{i}\arrow{s,l}{\pi_{\wH}}\node{\wh
G}\arrow{s,l}{\pi_{\wh G}}
\arrow{e,t}{f}\node{\CU}\arrow{s,r}{\pi_{\CU}}\arrow{e}\node{1}\\
\node{0}\arrow{e}\node{\wH/\wH^{00}}\arrow{e,t}{\wh i}\node{\wh G/\wh
G^{00}}\arrow{e,t}{\wh f}\node{\CU/\CU^{00}}\arrow{e}\node{0}
\end{diagram}
\end{equation}
As in the proof of \cite[Proposition 3.8]{ep-defquot}, the map $\wh f$ is continuous.

Assume now that $X\sub \wh G$ is a definable set of dimension smaller than $\dim \wh
G$. We want to show that $\pi_{\wh G}(X)$ has measure $0$. We are going to use
several variations of Fubini's theorem so let us see that the setting is correct. By
\cite{ep-defquot}, the group $\wh G/\wh G^{00}$ is isomorphic to $\R^k\times \R^n$
and the bottom sequence in the above diagram is just
\begin{equation}
\begin{diagram}
\node{0}\arrow{e}\node{\R^k}\arrow{e,t}{\wh i}\node{\wh G/\wh G^{00}}\arrow{e,t}{\wh
f}\node{\R^n}\arrow{e}\node{0}
\end{diagram}
\end{equation}
The above sequence necessarily splits as a Lie group, so by Fubini, a set $Y\sub \wh
G/\wh G^{00}$ has measure zero if and only if the set
$$\{u\in \R^n: m_{\R^k}(\wh f^{-1}(u)\cap Y)>0\}$$ has measure zero in $\R^n$. (By $m_{\R^k}(\wh f^{-1}(u)\cap Y)$ we mean the measure after identifying  $\R^k\times \{u\}$
with $\R^k$)

We
are now ready to start the proof.
\\

\noindent{\bf Case 1}  $\dim f(X)<\dim \CU$.
\\

Here we use Compact Domination in expansions of real closed fields (see \cite{HP}),
so by an earlier observation, $\CU$ also satisfies it. Hence, we have
$m(\pi_{\CU}(f(X)))=0$, and therefore, by the commutation of the above diagram and
Fubini we must have $m(\pi_{\wh G}(X))=0$.
\\

Most of the work goes towards the proof of the second case. For simplicity, let us
assume that $\wH\sub \wh G$.
\\

 \noindent{\bf Case 2}  $\dim
f(X)=\dim \CU$.
\\

We first establish two preliminary results.
\\

 \noindent{\bf Claim} We may assume that
$\lgdim(X)<k=\lgdim(\wh G)$.
\\

Indeed, by Lemma \ref{def-long}, we can decompose $f(X)$ into two definable sets
$Y_1\cup Y_2$ such that for every $u\in Y_1$, we have $\lgdim(f^{-1}(u)\cap X)<k$
and for every $u\in Y_2$, $\lgdim(f^{-1}(u)\cap X)=k=\dim(f^{-1}(u))$. Because $\dim
X<\dim \wh G$ and $\dim f(X)=\dim \CU$, the dimension of $Y_2$ must be smaller than
$\dim \CU$. By Case (1), we can ignore $Y_2$ and assume now that for every $u\in
f(X)$, $\lgdim(f^{-1}(u)\cap X)<k$. Since $\CU$ is short, it follows from
\cite{el-sbd} that $\lgdim(X)<k$.
\\

In the rest of the argument we prove the more general statement:
\\

\begin{lemma}\label{CD-short} If $X\sub \wh G$ is definable and $lgdim(X)<k$ then the measure of $\pi_{\wh
G}(X)$ is zero.
\end{lemma}
\proof
 We first prove a
result for the group $\wH$. By \cite[Proposition 3.8]{ep-defquot}, the group
$\wH/\wH^{00}$, equipped with the logic topology, is isomorphic to $\R^k$.

\begin{claim}\label{CD-H} If $Y\sub \wH$ is definable and $\lgdim(Y)<k$ then
$m(\pi_{\wH}(Y))=0$.
\end{claim}
\proof Recall that  $\wH$ is a subgroup of  $\la M^k,+\ra$ and that the set of all
short elements of $M^k$ is contained in $\wH^{00}$. Hence, if $B$ is any definably
connected short set, then $\pi_{\wH}(B)=\{b\}$ is a singleton.

The set $Y$ is a finite union of $m$-long cones, with $m<k$, hence we may assume
that $Y$ is such a cone  $C=B+\la C\ra$, where $\la C\ra=\left\{\sum_{i=1}^k
\lambda_i(t_i) : t_i \in I_i \right\}$, for long $I_i=(-a_i,a_i)$ and partial linear
maps $\lambda_i:I_i\to M^k$. We have
$$\pi_{\wH}(C)=b+\sum_{i=1}^m \pi_{\wH}(\lambda_i(t_i)).$$

Because $\pi_{\wH}$ is a homomorphism from $\la\wH,+\ra$ onto $\la \R^k,+\ra$, it
follows that for each $i=1,\ldots, m$, $t_i\mapsto \pi_{\wH}(\lambda_i(t_i))$ is a
partial homomorphism from $I_i$ into $\la \R^k,+\ra$. Hence, the image of the $\wh
G$-linear set  $\{\lambda_i(t):t\in I_i\}$ is a closed affine subset of $\R^k$ of
dimension $m$. Since $m<k$ we have $m(\pi_{\wH}(Y))=m(\pi_{\wH}(C))=0$.\qed
\\

\begin{claim}\label{section}
There exists a definable set $U_0\sub \CU$ with $\CU^{00}\sub U_0$, and a definable
section $s:U_0\to \wh G$ (i.e. $f s(u)=u$ for every $u\in U_0$), such that (i) the
function $s$ is continuous with respect to the topologies induced by $\CU$ and $\wh
G$ and (ii) $s(\CU^{00})\sub \wh G^{00}$.
\end{claim}
\proof Let $U_1\sub \CU$ be a definable generic set. By definable choice. there
exists a definable partial section $s:U_1\to \wh G$, namely, $sf(u)=u$ for all $u\in
U_1$. The map $s$ is piecewise continuous (with respect to the $\tau$-topologies of
 $\CU$ and $\hat G$) and therefore $U_1$ has a definable, definably connected $U_0\sub U_1$, still
generic in $\wh G$ such that $s$ is continuous on $U_0$. Using Compact Domination
for $\CU$, it follows from \cite[Claim 3, p.590]{HPP} that the set $U_0$ contains a
coset of $\CU^{00}$ so we may assume after translation in $U$ that $U_0$ contains
$\CU^{00}$ and $s:U_0\to \wh G $ is continuous. We may also assume that $s(0)=0$.

It is left to see that $s(\CU^{00})$ is contained in $\wh G^{00}$.  Consider the map
$\sigma(x,y)=s(x-y)-(s(x)-s(y))$,  a definable and continuous map from $U_0\times
U_0$ into $\wH$. Because the group topology on $\wH$ is the subspace topology of
$M^k$ and because $U_0\times U_0$ is a short definably connected set its image under
$\sigma$ is a short, definably connected subset of $\wH$ containing $0$. As we
pointed out earlier, it must therefore be contained in $\wH^{00}$.

We consider the set $\wh G_1=s(\CU^{00})+\wH^{00}$ and claim that $\wh G_1=\wh
G^{00}$. To see first that $\wh G_1$ is a subgroup, we note that
$$(s(u_1)+h_1)-(s(u_2)+h_2)= s(u_1-u_2)+(h_1+h_2-\sigma(u_1,u_2)).$$ When
$u_1,u_2\in \CU^{00}$ we have $\sigma(u_1,u_2)\in \wH^{00}$ and hence this sum is
also in $\wh G_1$.
 Because $s$ is definable and both $\CU^{00}$ and $\wH^{00}$ are type-definable the
group $\wh G_1$ is also type-definable. Because $\CU^{00}$ has bounded index in
$\CU$ and $\wH^{00}$ has bounded index in $\wH$ it follows that $\wh G_1$ has
bounded index in $\wh G$.
Since $\wh G_1$ is torsion-fee it follows from \cite[Proposition 3.6]{ep-defquot}
that $\wh G_1=\wh G^{00}$. In particular, $s(\CU^{00})$ is contained in $\wh
G^{00}$. \qed

Our goal is to show that $m(\pi_{\wh G}(X))=0$. By Fubini, it is sufficient to show
that for every $u\in \CU/\CU^{00}$, the fiber $\pi_{\wh G}(X)\cap \wh f^{-1}(u)$ has
zero measure in the sense of $\wH/\wH^{00}$. Namely, it is the translate in $\wh
G/\wh G^{00}$ of a zero measure subset of $\wH/\wH^{00}$.
\\

\noindent {\bf Claim } Fix $u\in \CU/\CU^{00}$. Then there exists a a definable set
$Y\sub \wH$ with $\lgdim(Y)<k$, and an element $g\in \wh G$ such that the fiber
$\pi_{\wh G}(X)\cap \wh f^{-1}(u)$ is contained in the set
$$\pi_{\wh H}(Y)+\pi_{\wh G}(g).$$

\noindent {\em Proof.} Fix $\bar u\in \CU$ such that $\pi_{\CU}(\bar u)=u$. By
translation in $\wh G$ and in $\CU$ we may assume that the domain of the partial
section $s$ which was defined above, call it still $U_0$,  contains $\bar
u+\CU^{00}$. If we let $g=s(\bar u)$ then $s(\bar u+\CU^{00})\sub g+\wh G^{00}$.

Consider the definable map $x\mapsto x-s(f(x))$ from $X\cap f^{-1}(U_0)$ into $\wH$
and let $Y$ be its image. Because $\lgdim(X)<k$, we must also have $\lgdim(Y)<k$.
We claim that this is the desired $Y$. Indeed, we assume that $\wh f(\pi_{\wh
G}(x))=u$ for some $x\in X$ and show that $\pi_{\wh G}(x)\in \pi_{\wh H}(Y)+\pi_{\wh
G}(g)$.

By the commuting diagram above, $f(x)\in \bar u+\CU^{00}\sub U_0$ and therefore
$x-s(f(x))\in Y$. Since $s(\bar u+\CU^{00})$ is contained in $s(\bar u)+\wh G^{00}$,
we also have $s(f(x))\in g+\wh G^{00}$. We now have
$$\pi_{\wh G}(x)=\pi_{\wh G}(x-sf(x))+\pi_{\wh G}(sf(x))\in \pi_{\wh H}(Y)+\pi_{\wh
G}(g).$$\qed

We can now complete the proof that $\pi_{\wh G}(X)$ has measure $0$. For every $u\in
\CU/\CU^{00}$ we find a definable $Y\sub \wH$ as above. By Claim \ref{CD-H}, the set
$\pi_{\wH}(Y)$ has measure $0$ in $\wH/\wH^{00}$, hence the fiber $\pi_{\wh
G}(X)\cap \wh f^{-1}(u)$ is a translate of a measure zero subset of $\wH/\wH^{00}$.
By Fubini the measure of $\pi_{\wh G}(X)$ is zero.

This  ends the proof of Lemma \ref{CD-short} and with it that of Compact Domination for
abelian $G$. \vskip.2cm

\noindent {\bf II. The general case ($G$ not necessarily abelian).}\vskip.2cm

Assume now that $G$ is an arbitrary definably compact group. By \cite{HPP2}, $G$ is
the almost direct product of a definably connected abelian group $G_0$ and a
definable semi-simple group $S$. It is enough to prove the result for a finite cover
of $G$ hence we may assume that $G=G_0\times S$. By \cite[Theorem 4.4 (ii)]{HPP2},
the group $S$ is definably isomorphic to a semialgebraic group over a definable real
closed field so it must be short, and therefore $\lgdim G=\lgdim G_0=k$. To simplify
the diagram, we use $\overline{G_0}=G_0/G_0^{00}$, $\overline{S}=S/S^{00}$, so we
have $G/G^{00}=\overline{G_0}\times \overline{S}$.

We have
$$
\begin{diagram}
\node{0}\arrow{e}\node{G_0}\arrow{e,t}{i}\arrow{s,l}{\pi_{G_0}}\node{G_0\times
S}\arrow{s,l}{\pi_{G}}
\arrow{e,t}{f}\node{S}\arrow{s,r}{\pi_S}\arrow{e}\node{0}\\
\node{0}\arrow{e}\node{\overline{G_0}}\arrow{e,t}{\wh i}\node{\overline{G_0}\times
\overline{S}}\arrow{e,t}{\wh f}\node{\overline{S}}\arrow{e}\node{0}
\end{diagram}
$$

 Assume now that
$X\sub G$ is a definable set and $\dim(X)<\dim(G)$. If $\dim(f(X))<\dim(S)$ then by
Compact Domination in expansions of fields, the Haar measure of $\pi_S(f(X))$ in
$\overline{S}$ is $0$ and therefore $m(\pi_G(X))$ in $G/G^{00}$ is $0$.

If $\dim(X)=\dim(S)$ then, as in the abelian case, we may assume, after partition,
that for every $s\in S$, $\lgdim(f^{-1}(s)\cap X)<k$. Because $S$ is short, it
follows that $\lgdim(X)<k$ and therefore the projection of $X$ into $G_0$, call it
$X'$, has long dimension smaller than $k$. But now, by Lemma \ref{CD-short}, the
Haar measure in $\overline{G_0}$ of $\pi_{G_0}(X')$ equals to $0$. By Fubini,  the
Haar measure of $\pi_G(X)$ must also be zero.

This ends the proof of Compact Domination for definably compact groups in o-minimal
expansions of ordered groups.\qed

\section{Appendix A - pullback and pushout}

\subsection{Pushout}
\proof[Proof of Proposition \ref{pushout1}]

We start with
$$\begin{diagram}
\node{A}\arrow{e,t}{\al}\arrow{s,l}{\beta}\node{B}\\
\node{C}
\end{diagram}$$
and prove the existence of the pushout $D$.
 We first review the standard construction of $D$ (without verifying the algebraic facts).
 We consider the direct product $B\times C$ and take $D=(B\times C)/H$ where
$H$ is the subgroup $H=\{(\alpha(a),-\beta(a)):a\in A\}$. If we denote by $[b,c]$
the coset of $(b,c)$ mod $H$ then the maps $\gamma,\delta$ are defined by
$\gamma(b)=[b,0]$ and $\delta(c)=[0,c]$. Assume now that we also have
$$\begin{diagram}
\node{A}\arrow{e,t}{\al}\arrow{s,l}{\beta}\node{B}\arrow{s,r}{\gamma'}\\
\node{C}\arrow{e,t}{\delta'}\node{D'}
\end{diagram}$$ We define $\phi:D\to D'$ by
$\phi([b,c])=\gamma'(b)+\delta'(c)$. Clearly, if all data are definable then so are
$B\times C$ and $H$,  and therefore, using definable choice, $D$ and the associated
maps are definable.

 If $\al$ is injective then $\delta$ is also injective, and if $\beta$ is surjective then so is $\gamma$ (see observation (b) on p.
53 in \cite{Fuchs})

Suppose that $A,B,C$ and $\al,\beta$ are $\bigvee$-definable and that $\al(A)$ is
compatible subgroup of $B$. Clearly $B\times C$ is $\bigvee$-definable and it is
easy to see that $H$ is a $\bigvee$-definable subgroup. We want to show that $H$ is
a compatible subgroup of $B\times C$. For that we write $A=\bigcup A_i$, $B=\bigcup
B_j$ and $C=\bigcup C_k$. It follows that $B\times C=\bigcup_{j,k}B_j\times C_k$. To
show compatibility of $H$ it is enough to show that for every $j,k$, the
intersection $(B_j\times C_k)\cap H$ is definable. Because $\alpha(A)$ is compatible
in $B$, the set $B_j\cap \al(A)$ is definable. Hence, there is some $i_0$ such that
$\al(A_{i_0})\supseteq B_j\cap \al(A)$. Moreover, because $\al$ is injective
$\al^{-1}(B_j)\subset A_{i_0}$.
 It follows that the
intersection $H\cap (B_j\times C_k)$ equals
$$ \{(\al(a), -\beta(a))\in B_j\times C_k:a\in A\}=\{(\al(a),-\beta(a))\in B_j\times
C_k:a\in A_{i_0}\}.$$
 The set on the right is clearly definable, hence $H$ is a
compatible subgroup of $B\times C$, so $D=(B\times C)/H$ is $\bigvee$-definable (see
Fact \ref{edmundo}). It is now easy to check that $\gamma:B\to D $ and $\delta: C\to D$
are $\bigvee$-definable.

If $E=B/\al(A)$ then, by the compatibility of $\al(A)$, we see that $E$ is
$\bigvee$-definable. If $\pi:B\to E$ is the projection then we define $\pi':D\to E$
by $\pi'([b,c])=\pi(b)$. It is routine to verify that $\pi'$ is a well-defined
surjective homomorphism whose kernel is $\delta(C)$. It follows, using Fact
\ref{edmundo}, that $\delta(C)$ is a compatible subgroup of $D$. Finally, it is
routine to verify commutation of all maps.\qed

\proof[Proof of Lemma \ref{pushout2}] We have \begin{equation}
\begin{diagram}
\node{B}\arrow{e,t}{\gamma} \node{D}\arrow{e,t}{\mu}\node{F}\\
 \node{A}\arrow{n,l}{\alpha}\arrow{e,t}{\beta}\node{C}\arrow{n,r}{\delta}\arrow{e,t}{\eta}\node{E}\arrow{n,r}{\xi}
\end{diagram}
\end{equation}
with $D$ the pushout of $B$ and $C$ over $A$ and $F$ the pushout of $B$ and $E$ over
$A$ and we want to see that $F$ is also the pushout of $D$ and $E$ over $C$.

It is sufficient to show that for every given commutative diagram
\begin{equation} \begin{diagram}
 \node{D}\arrow{e,t}{\mu'}\node{F'}\\
\node{C}\arrow{n,r}{\delta}\arrow{e,t}{\eta}\node{E}\arrow{n,r}{\xi'}
\end{diagram}
\end{equation}
there is a map $\phi':F\to F'$ such that $\phi'\mu=\mu'$ and $\phi' \xi=\xi'$
(according to the definition we also need to prove uniqueness but this follows).

By commutativity we have $\mu'\delta=\xi'\eta$ and hence $\mu'\delta
\beta=\xi'\eta\beta$. Since $\delta\beta=\gamma\alpha$ we also have
$(\mu'\gamma)\alpha=(\xi')\eta\beta$. We now use the fact that $F$ is the pushout of
$B$ and $E$ over $A$ and conclude that there is $\phi':F\to F'$ such that
\begin{equation}\label{app1}
(i) \phi'\xi=\xi'\,\,\, \mbox{ and } \,\,\, (ii) \phi'\mu\gamma=\mu'
\gamma\end{equation} (i) gives half of what we need to show so it is left to see
that $\phi'\mu=\mu'$.

Consider the commutative diagram
\begin{equation}\begin{diagram}
\node{B}\arrow{e,t}{\mu'\gamma}\node{F'}\\
\node{A}\arrow{n,l}{\alpha}\arrow{e,t}{\beta}\node{C}\arrow{n,r}{\xi'\eta}
\end{diagram}
\end{equation}
Because $D$ is the pushout of $B$ and $C$ over $A$, there is a unique map $\psi:D\to
F'$ with the property
$$(i)\psi \delta=\xi'\eta\,\,\, \mbox{ and }\,\,\,(ii) \psi\gamma=\mu'\gamma.$$

If we can show that both maps $\mu'$ and $\phi'\mu$ from $D$ into $F'$ satisfy these
properties of $\psi$ then by uniqueness we will get their equality. For $\psi=\mu'$,
(i) is part of the assumptions, and (ii) is obvious. For $\psi=\phi'\mu$, we obtain
(ii) directly from (\ref{app1})(ii). To see (i), start from (\ref{app1})(i),
$\phi'\xi=\xi'$, and conclude $\phi'\xi\eta=\xi'\eta$. By commutation,
$\xi\eta=\mu\delta$ so we obtain $\phi'\mu\delta=\xi'\eta$, as needed. We therefore
conclude that $\mu'=\phi'\mu$ and hence $F$ is the pushout of $E$ and $D$ over
$C$.\qed

\subsection{Pullback}

\proof[Proof of Proposition \ref{pullback1}]  Consider the diagram
$$\begin{diagram}
\node{} \node{B}\arrow{s,r}{\al}\\
\node{C}\arrow{e,b}{\beta}\node{A}
\end{diagram}$$

We again review the algebraic construction of a pullback (which is simpler because
we take no quotients). We let  $$D=\{(b,c)\in B\times C:\al(b)=\beta(c)\},$$ and the
maps are just $\gamma(b,c)=b$ and $\delta(b,c)=c$. Given
$$\begin{diagram}
\node{D'}\arrow{e,t}{\gamma'}\arrow{s,l}{\delta'} \node{B}\arrow{s,r}{\al}\\
\node{C}\arrow{e,b}{\beta}\node{A}
\end{diagram}$$ we define
$\phi(d')=(\gamma'(d'),\delta'(d'))\in D$.

Clearly, if all data are definable then so is $D$ and the associated maps. Similarly,
if all data are $\bigvee$-definable then so are $D$ and the associated maps.

If $G=\ker(\gamma)$ then
$$G=\{(b,c)\in D:b=0\}=\{(0,c)\in B\times C:\beta(c)=0\},$$ and then clearly
$j(0,c)=c$ is an isomorphism of $G$ and $H=\ker(\beta)$. If all given data are
$\bigvee$-definable then so are $G,H$ and the associated maps. Furthermore, since
$G$ and $H$ are kernels of $\bigvee$-definable maps they are clearly compatible in
$D,C$, respectively.

If $\beta$ is surjective then so is $\gamma$ and the sequences in the diagram are
exact (and the diagram is commutative).\qed

\section{Appendix B - Short and long set}
{\em We assume that $\CM$ is an o-minimal semi-bounded expansion of an ordered group}

\begin{lemma}\label{def-long} Let $S\sub M^r$ be a definable short set and let $A\sub S\times M^n$
be a definable set. For $s\in S$, we let $A_s=\{x\in M^n:(s,x)\in A\}$. Then, for
every $\ell\geq 0$, the set $\ell(A)=\{s\in S:\lgdim(A_s)=\ell \}$ is definable.
\end{lemma}
\proof By \cite{el-sbd}, the set $A$ can be written as a union of long cones
$\bigcup C_i$. Since $\lgdim(X_1\cup \cdots\cup X_m)=max_i(\lgdim(X_i))$, we may
assume that $A$ itself is a long cone $A=B+\sum_{i=1}^k \lambda_i(t_i)$, where $B\sub
M^{r+n}$ is a short cell, $\lambda_1,\ldots,\lambda_k$ are $M$-independent
partial linear maps $\lambda_i:I_i\to M^{r+n}$ and $I_i=(0,a_i)$ are long
intervals. We write $\lambda_i=(\lambda_i^1,\ldots,\lambda_i^{r+n})$, for
$i=1,\ldots, k,$ so each $\lambda_i^j$ is a partial endomorphism from $I_i$ into
$M$.

We claim that for every $s\in S$, $\lgdim(A_s)=k$. This clearly implies what we
need.

For $b=(b_1,\ldots, b_{r+n})\in B$, $i=1,\ldots,k$ and $t_i\in I_i$, we have
$b_i+ \lambda_i(t_i):I_i\to A$. Therefore, we have
$(b_1,\ldots,b_r)+(\lambda^1_i(t_i),\ldots,\lambda^r_i(t_i))\in S$. Each
$\lambda^j_i$ is either injective or constantly $0$ and hence, because $S$ is short
and each $I_i$ is long,  for each $j=1,\ldots,r$ and $i=1,\ldots,k$, we have
$\lambda^j_i \equiv 0$. It follows that for every $b\in B$, we have $(b_1,\ldots,
b_r)\in S$.

For $i=1\ldots, k$, we let
$$\hat \lambda_i=(\lambda_i^{r+1},\ldots,
\lambda_i^{r+n}):I_i\to M^n.$$ Because $\lambda_1,\ldots,\lambda_k$ were
$M$-independent, it is still true that $\hat \lambda_1,\ldots,\hat\lambda_k$ are
$M$-independent. We now have, for every $s\in S$,
$$A_s=\left\{b+\sum_{i=1}^k\hat \lambda_i (t_i) : b\in B_s, t\in I_i\right\}$$
and therefore the set $A_s$ is a $k$-long cone, so $\lgdim(A_s)=k$.\qed

\end{document}